\documentclass[a4paper, reqno]{amsart}

\title[{}]{Hypercubic structures behind $\hat{Z}$-invariants}

\author{Shoma Sugimoto}
\address{Yau Mathematical Sciences Center, Tsinghua University}
\email{shomasugimoto361@gmail.com}

\usepackage{latexsym}
\usepackage{amsmath}
\usepackage{amsfonts}
\usepackage{amssymb}
\usepackage{amsxtra}
\usepackage{mathrsfs}
\usepackage{eucal}
\usepackage[all]{xy}
\usepackage{color}
\usepackage{braket}
\usepackage{graphics, setspace}
\usepackage{etoolbox, textcomp}
\usepackage{comment}
\usepackage{changepage}
\usepackage{xcolor}
\definecolor{rouge}{rgb}{0.85,0.1,.4}
\definecolor{bleu}{rgb}{0.1,0.2,0.9}
\definecolor{violet}{rgb}{0.7,0,0.8}
\usepackage[colorlinks=true,linkcolor=bleu,urlcolor=violet,citecolor=rouge]{hyperref}
\usepackage{cleveref}
\usepackage{lscape}
\usepackage{tikz}
\usepackage{amsthm}
\usetikzlibrary{graphs}
\usetikzlibrary {arrows.meta}
\usepackage{enumitem}

\theoremstyle{definition}
\newtheorem{definition}{Definition}[section]

\newtheorem{theorem}[definition]{Theorem}
\newtheorem{corollary}[definition]{Corollary}
\newtheorem{lemma}[definition]{Lemma}

\newtheorem{remark}[definition]{Remark}

\newtheorem*{fundamental-problem}{\underline{Fundamental Problem (FP)}}

\newtheorem*{guiding-principle}{\underline{Fundamental Problem from the new perspective}}

\numberwithin{equation}{section}

\newcommand{\Z}{\mathbb{Z}}

\newcommand{\C}{\mathbb{C}}

\newcommand{\ch}{\operatorname{ch}}

\newcommand{\g}{\mathfrak{g}}

\newcommand{\sll}{\mathfrak{sl}}
\newcommand{\osp}{\mathfrak{osp}}

\newcommand{\SL}{SL_2}

\newcommand{\Bits}[1]{\Z_2^{#1}}

\newcommand{\Left}[1]{\mathcal{#1}^{\operatorname{left}}}

\newcommand{\Right}[1]{\mathcal{#1}^{\operatorname{right}}}

\newcommand{\leftop}[1]{\operatorname{ltop}(\mathcal{#1})}
\newcommand{\rightop}[1]{\operatorname{rtop}(\mathcal{#1})}
\newcommand{\gap}[1]{\operatorname{gap}(\mathcal{#1})}

\begin{document}
\maketitle

\begin{abstract}
We propose an abelian categorification of \textit{$\hat{Z}$-invariants} \cite{GPPV} for Seifert $3$-manifolds.
First, we give a recursive combinatorial derivation of these $\hat{Z}$-invariants using graphs with certain hypercubic structures.
Next, we consider such graphs as annotated Loewy diagrams in an abelian category, allowing non-split extensions by the ambiguity of embedding of subobjects.
If such an extension has good algebraic group actions, then the above derivation of $\hat{Z}$-invariants in the Grothendieck group of the abelian category can be understood in terms of the theory of \textit{shift systems} \cite{LS}, i.e., Weyl-type character formula of the \textit{nested} Feigin-Tipunin constructions.
For the project of developing the dictionary between logarithmic CFTs and 3-manifolds \cite{CFGHP}, these discussions give a glimpse of a hypothetical and prototypical, but unified construction/research method for the former from the new perspective, \textit{reductions of representation theories by recursive structures}.
\end{abstract}

\tableofcontents

\section{Introduction}

\subsection{Background and Motivation}
\label{subsection: background, motivation and program}
The study of \textit{logarithmic conformal field theory} (\textit{LCFT}, or \textit{logarithmic/irrational vertex (operator) algebra}) has become increasingly important in both mathematics and physics.
As an example, in $2017$, Gukov-Pei-Putrov-Vafa \cite{GPPV} introduced the \textit{$\hat{Z}$-invariant} for plumbed graphs. 
This is a (plumbed) 3-manifold topological invariant that takes values in the $q$-series and converges to the Witten-Reshetikhin-Turaev invariant at roots of unity \cite{Mur3} and has been found to give rich examples of ``spoiled" modular forms, such as mock/false theta functions. 
Since the characters of rational CFT are closed under modular transformations \cite{Zhu}, if CFTs can be constructed that would have $\hat{Z}$-invariants as their characters, then they should give rich and interesting examples of LCFT. 
Indeed, in \cite{CCFGH,CFGHP,CCKPS} etc., for various reasons they have developed such a ``dictionary" between 3-manifolds and LCFTs. 
Our goal is to provide a \textit{new mathematical approach/perspective} to their project, which seeks to link LCFT, quantum topology, number theory, and $3$d-$3$d correspondence in physics etc.
Let us raise the fundamental problem.

\begin{fundamental-problem}
For any simple Lie algebra $\g$,  nilpotent element $f\in\g$, and plumbed graph $\Gamma$, give a \textbf{unified construction and research method} for LCFTs such that their characters almost give the (W-algebraic variations of the) corresponding $\hat{Z}$-invariants (see also Remark \ref{remark: restriction of scope in the present paper} below).
\end{fundamental-problem}

However, at this time not much is known about  mathematical constructions and research methods of LCFTs. 
Indeed, at least in the mathematical level, the above correspondence is known only for the two easiest examples, i.e. the cases of $3$- or $4$-legs star graphs (the \textit{triplet Virasoro VOA} at level $\tfrac{1}{p_1}$ or $\tfrac{p_2}{p_1}$ for $p_1,p_2\in\Z_{\geq 2}$, respectively.
In \cite{CCFGH,CFGHP}, they proved that the characters of these LCFTs almost gives the corresponding $\hat{Z}$-invariants).
One of the reasons why the study of LCFTs is more difficult than rational cases is that the VOA structures are complicated, making it difficult to use algebraic and straightforward approaches.
For example, the two easiest cases above have long been studied only for $\g=\sll_2$, and no results were published for general simple Lie algebra $\g$ for many years until \cite{S1}.
Perhaps due to this lack of examples and research methods for LCFTs, in the mathematical level, studies of $\hat{Z}$-invariants and the above ``dictionary" has so far been mainly concerned with the level of $q$-series.
In contrast, our approach focuses not on the $q$-series level, but rather on the construction/research method of LCFTs behind $\hat{Z}$-invariants.
In other words, it is intended to provide a bridge between the world of $q$-series and that of LCFTs.
Let us explain the detail in the following two paragraphs.

\begin{remark}
\label{remark: restriction of scope in the present paper}
Throughout this paper, we sometimes use the phrase ``$X$ \textit{almost} gives $Y$". 
This means that they will match except for minor differences, such as $q$-power, Dedekind eta function, taking linear combinations, etc. 
Also, the author have in mind the $\hat{Z}$-invariants of the \textit{negative Seifert/false theta} side, and cannot be certain at this time about the \textit{positive Seifert/mock theta} side or more complicated cases, because the corresponding VOAs have not yet been mathematically well-studied.
However, our purpose here is not to give an elaborate correspondences on $q$-series or to go deep into the modern studies on modularity, but to give a direction/object that seems plausible and interesting for the mathematically almost unknown problem of the correspondence between $\hat{Z}$-invariants and LCFTs.
Thus, we consider such easiest cases and relatively rough correspondences at the level of $q$-series, and leave more precise correspondences and deep modularity theorems for future works.
\end{remark}

In $2010$, Feigin-Tipunin \cite{FT} conjectured that the \textit{multiplet (principal) W-algebra} (i.e. higher rank version of the triplet Virasoro VOA) associated with a simply-laced simple Lie algebra $\g$ has a certain geometric construction called \textit{Feigin-Tipunin (FT) construction}, i.e. the space of global sections 
\begin{align*}
H^0(G\times_BV)  
\end{align*}
of a VOA-bundle $G\times_BV$ over the flag variety $G/B$, and that a certain analogue of the Borel-Weil-Bott (BWB) theorem holds. 
The author proved their conjectures in \cite{S1,S2} and used them to prove some basic properties of the corresponding multiplet principal W-algebras, such as Weyl-type character formulas, decomposition theorem, and construction of irreducible modules, etc. 
Similar results were also extended to the $V^{(p)}$-algebra (i.e. doublet/triplet affine $\sll_2$) in \cite{CNS} by combining the method in \cite{S1,S2} with the inverse quantum Hamiltonian reduction.
On the other hand, it is \textit{not} the VOA structures that plays an essential role in the discussion of \cite{S1,S2}, but rather a certain hidden \textit{Lie algebraic pattern} (Although the situation is a bit more complicated in \cite{CNS}, the main flow of the discussion itself is still the same because such a Lie algebraic pattern is preserved under the various operations on VOAs under consideration, such as (inverse) quantum Hamiltonian reductions). 
In \cite{LS}, the author extracted such a hidden Lie algebraic pattern and named \textit{shift system}, and proved that most of the arguments in \cite{S1,S2} are valid even under such an \textit{abstracted purely Lie algebraic setting}. 
Let us explain on it a bit more (for more detail, see \cite[Introduction]{LS}).
The shift system consists of the triple
\begin{align*}
(\Lambda,\uparrow,\{V_\lambda\}_{\lambda\in\Lambda})
\end{align*}
of a parameter $W$-module $\Lambda$, a shift map $\uparrow\colon W\times\Lambda\rightarrow P$, and the family of weight $B$-modules $\{V_\lambda\}_{\lambda\in\Lambda}$ with a few simple axioms (see \cite[Definition 1.6]{LS}).
Note that the definition is purely Lie algebraic and does not need any VOA structures.
If an ``input data” $B$-module $V$ fits into a shift system (i.e., satisfies the axiom of shift system), then the main results of \cite{S1,S2} automatically hold for the ``output data” $H^0(G\times_BV)$ which is isomorphic to the \textit{maximal $G$-submodule} of $V$ in our case 
(Note that in the application to LCFT, $V$ is a simpler CFT such as free-field algebra, and unlike LCFT, algebraic/straightforward approaches are often valid (e.g., reduction to rank $1$ by orthogonal decompositions). 
In other words, the shift system theory is a machinery that \textit{reduces} the study of LCFTs to that of simpler CFTs.
Indeed, such a verification of the axiom was done step-by-step and in detail in \cite[Section 2.4]{LS} for the cases where $V$ is the positive definite rescaled root lattice VOA $V_{\sqrt{p}Q}$, resulting in most results in \cite{S1,S2} being extended to any simple Lie algebra $\g$ and Lie superalgebra $\osp(1|2n)$, respectively). 
In particular, since the BWB-type theorem holds, the \textit{Weyl-type character formula} 
\begin{align}
\label{general Weyl character formula in introduction}
\ch H^0(G\times_BV_\lambda)
=
\sum_{\beta\in P_+}\dim L_{\beta}
\sum_{\sigma\in W}
(-1)^{l(\sigma)}
\ch V_{\sigma\ast\lambda}^{h=\beta-\sigma\uparrow\lambda}
\end{align}
follows as its corollary, so that the character of $H^0(G\times_BV_\lambda)$ is \textit{reduced} to that of $V_{\sigma\ast\lambda}$. 
These results suggest that, at least in some aspects of the mathematical study of LCFT, one should develop the \textit{geometric representation theory} based on certain Lie algebraic patterns behind various LCFTs in common, rather than directly examining the complicated VOA structures of individual cases algebraically.
In addition, the conditions of the BWB-type theorem of FT constructions above are very similar to those of the \textit{modular representation theory} (cf. \cite{Jan}), which suggests that the geometric representation theory of FT constructions may also be effective for studying the \textit{log Kazhdan-Lusztig correspondence}, i.e. expected categorical equivalences between LCFTs and \textit{quantum groups}. 
Therefore, this theory would be a candidate for the \textbf{``unified research method"} in the FP above.

\begin{remark}
\label{remark: internal structure vs external structure}
Let us explain the significance of the shift system theory from the perspective of the study of \textit{$W$-algebras}.
Since $H^0(G\times_BV_\lambda)$ has the natural $G$-action, we have the decomposition
\begin{align*}
H^0(G\times_BV_\lambda)
\simeq
\bigoplus_{\beta\in P_+}L_\beta\otimes\mathcal{W}_{-\beta+\lambda}
\end{align*}
as a $G$-module, where $L_\beta$ is the finite-dimensional irreducible $G$-module with highest weight $\beta$ and $\mathcal{W}_{-\beta+\lambda}$ is the multiplicity of a weight vector of $L_\beta$.
If $H^0(G\times_BV_0)$ has a V(O)A structure, the $G$-orbifold $\mathcal{W}_0=H^0(G\times_BV_0)^G$ is the vertex (operator) subalgebra of $H^0(G\times_BV_0)$.
This is the W-algebra in our case, and each $\mathcal{W}_{-\beta+\lambda}$ is a 
$\mathcal{W}_0$-module (in other words, the multiplet W-algebra $H^0(G\times_BV_0)$ can be regarded as an infinite extension of $\mathcal{W}_0$).
There has been a lot of works on W-algebras, but our cases have exceptional levels and existing results are often inapplicable.
However, in our cases, through the geometric representation theory for the extension $H^0(G\times_BV_\lambda)$, we can consequently obtain results on the $\mathcal{W}_0$ and its modules $\mathcal{W}_{-\beta+\lambda}$ (cf. the simplicity theorems of $\mathcal{W}_{-\beta+\lambda}$ in \cite{S2,LS,CNS}. 
In particular, in \cite{S2,LS}, the simplicity theorem of the Arakawa-Frenkel module at irrational levels \cite{ArF} is extended to our positive integer level cases).
This approach/perspective of obtaining information about $\mathcal{W}_{-\beta+\lambda}$ by 
focusing on the \textit{Lie algebraic patterns external to $\mathcal{W}_{-\beta+\lambda}$} (i.e., the FT construction and shift system), rather than the \textit{internal (VOA) structure of $\mathcal{W}_{-\beta+\lambda}$} which is \textit{non-essential and decorative} from the perspective of the Lie algebraic patterns (because this is just the \textit{multiplicity of a weight vector} of $L_\beta$)
is not only one demonstration of the above idea, but also plays a central role in this paper.
\end{remark}

On the other hand, all the examples \cite{S1,S2,LS,CNS} in the previous paragraph are multiplet W-algebras at positive integer levels, i.e. just W-algebraic variations of LCFTs corresponding to the easiest plumbed graph, $3$-legs star graph.
How can we construct and study LCFTs corresponding to general star graphs or plumbed graphs?
For this question, the author conceived the following idea a few years ago 
(In fact, his main motivation for introducing the shift system \cite{LS} (and for writing this paper) was to mathematically express this idea which was very roughly sketched in \cite{S3}).
Let $\tilde{V}:=\tilde{V}_{\lambda_1,\dots,\lambda_N}$ be a weight $B$-module which forms a shift system with respect to the first parameter $\lambda_1\in\Lambda_1$.
The FT construction $H^0(G\times_B\tilde{V})$ has a natural $G$-module structure, but suppose that there exists a suitable quotient $\tilde{H}^0(G\times_B\tilde{V})$ of $H^0(G\times_B\tilde{V})$, which in turn forms a shift system with respect to a $B$-action with the same $H$-action, and the second parameter $\lambda_2\in\Lambda_2$.
If such a situation holds for the parameters $\lambda_1,\dots,\lambda_N$ sequentially, that is, if we consider a \textit{recursive shift system} 
(i.e., 
a special kind of shift system $\tilde{V}:=\tilde{V}_{\lambda_1,\dots,\lambda_N}$ such that for each $0\leq n\leq N-1$, there exists a quotient
\begin{align}
\label{quotient of the nested FT constructions}
\underbrace{\tilde{H}^0(G\times_B\tilde{H}^0(G\times_B\cdots\tilde{H}^0(G\times_B}_{\text{$n$-times}}\tilde{V})\cdots))
\ \ \text{of} \ \ 
\underbrace{{H}^0(G\times_B\tilde{H}^0(G\times_B\cdots\tilde{H}^0(G\times_B}_{\text{$n$-times}}\tilde{V})\cdots))
\end{align}
such that the LHS forms a shift system with respect to the parameter $\lambda_{n+1}\in\Lambda_{n+1}$), 
then by taking the quotients of the global section functors $\tilde{H}^0(G\times_B-)$ repeatedly, we finally obtain the \textit{nested FT construction}
\begin{align}
\label{nested FT construction in general cases}
\underbrace{H^0(G\times_B\tilde{H}^0(G\times_B\cdots\tilde{H}^0(G\times_B}_{\text{$N$-times}}\tilde{V})\cdots)).
\end{align}
By construction, the shift system appears at each recursive step, so the geometric representation theory of FT constructions in the previous paragraph can be applied repeatedly.
In particular, if a suitable equality holds between the characters of both sides of \eqref{quotient of the nested FT constructions} for $0\leq n\leq N-1$, then by repeatedly applying the Weyl-type character formula \eqref{general Weyl character formula in introduction} recursively, the character of \eqref{nested FT construction in general cases} can finally be computed from $\ch\tilde{V}_{w_1\ast\lambda_1,\dots,w_N\ast\lambda_N}$ ($w_1,\dots,w_N\in W$).
Surprisingly, this almost gives the \textit{``bosonic formula" of the $\hat{Z}$-invariant} of the $(N+2)$-legs star graphs (More precisely, consider the subspace of the Cartan weight $0$).
In other words, we obtain a \textit{recursive computation algorithm} for the $\hat{Z}$-invariants which has a certain \textit{representation-theoretic interpretation}.
It suggests that by increasing the recursion depth (i.e. number of recursion steps required to reach the base case $\tilde{V}$), one would expect the LCFTs corresponding to general star graphs to be constructed, and further that their representation theories are to some extent controlled by those of the ``input data" $\tilde{V}$ and the geometric representation theory of FT constructions
(The cases of general plumbed graphs are more complicated, but the bosonic forms of the $\hat{Z}$-invariants (cf. \cite[Proposition 4.2]{Mur2}) imply that the basic idea would be still the same).
Therefore, the study of mechanisms that realize such recursive shift systems
is not only a candidate for the  \textbf{``unified construction"} in the FP
above, but also complementary to the ``unified research method" in the previous paragraph.
Also, note that these theories, hereafter collectively referred to as the \textit{shift system theory}, are strongly motivated by FP, but are \textit{independent of CFT} because they do not assume any V(O)A structure.
Rather, the abstraction of the V(O)A structure enables us to develop such a \textit{``hypothetical solution"} to FP up front, even though FP is far beyond our current mathematical understanding of LCFT.
It would be beneficial not only in making the logical structure of FP transparent (i.e. separating the parts where V(O)A structures are essential from the other), but also in improving accessibility for researchers with more general backgrounds.
Indeed, as explained in Section \ref{section: technical summary} below, the shift system theory can be regarded as an \textit{abelian categorization} of $\hat{Z}$-invariants, or the abelian categorical aspect of the module category of the corresponding LCFTs.

\subsection{Technical summary}
\label{section: technical summary}
In this paper, we take the first step toward the theory of recursive shift systems for $\g=\sll_2$.
Leaving the detailed and precise discussion to Section \ref{section: combinatorics} and \ref{section: Lie algebraic aspects}, this subsection provides the technical motivation and overview of this paper.
Although most of the techniques used in this paper are only elementary combinatorics, it might be useful to have diagrammatic images described below to understand the meaning.

The key condition in the shift system is the existence of short exact sequences of $B$-modules called \textit{Felder complexes} (see Section \ref{subsection: recursive shift systems}).
A Felder complex has the following diagrammatic representation (Felder diagram): 
First, Cartan weights $h$ and dimensions $d$ of the $B$-modules are aligned horizontally and vertically, respectively.
Then, according to the Felder complex (i.e., from sub to quotient), draw arrows between the coordinates $(h,d)$ representing the corresponding multiplicity spaces of weight vectors of the $B$-modules (these Felder diagrams look like \cite[Figure 4,5]{FGST}).
By taking $H^0(\SL\times_B-)$ (FT construction), the subdiagram corresponding to the maximal $\SL$-part is extracted.
Although the Felder diagram itself is simple, we consider more complicated situations.
That is, when the coordinates $(h,d)$ are not the smallest units in the discussion, but have their internal structures.
For example, by considering an \textit{annotated Loewy diagram} for an object $V$ of an abelian category (or, by Freyd-Mitchell embedding theorem, a module $V$ over a ring $R$), we have a \textit{directed acyclic graph (DAG)} that colored by the simple composition factors of $V$.
If $V$ is in a shift system with respect to a $B$-action commuting with the $R$-action, then the multiplicity space of each weight vector might not be a node of the DAG, but a more complicated annotated Loewy diagram.
However, as mentioned in Remark \ref{remark: internal structure vs external structure}, the internal structures of the multiplicity spaces are not important in the shift system theory, and it can be examined to some extent by the external Lie algebraic patterns.
Therefore, to construct a recursive shift system from given DAGs or corresponding objects, it is important to ``push" their complexity into the internal structures of the multiplicity spaces, so that the whole can be regarded as a Felder diagram.
As a recursive structure to realize such a situation repeatedly, this paper uses \textit{hypercube graphs} (recall that an $(n+1)$-cube graph is realized by moving an $n$-cube graph).
As described below, we aim to give a mechanism such that at $n$-th recursive step $(1\leq n\leq N)$, the multiplicity spaces have the structure of $2(N-n)$-cube graphs, and $H^0(\SL\times_B-)$ and the quotient in the recursive shift system can continue to be regarded as ``contractions" of the hypercube graphs from different directions, respectively.

Let us give an overview of this paper.
In Section \ref{section: combinatorics}, we first replace each coordinate of the Felder diagram above with a $2N$-cube graph $(\pm^N|\pm^N)$ appropriately colored by $C_N:=\Bits{N}\times\Z$ and \textit{``fragments"} it into $2^N$ DAGs $[\lambda|\pm^N]$ ($\lambda\in\Bits{N}$. For $N=2$, see \cite[Figure 2,3]{FGST}).
For this fragmented $2N$-cubic Felder diagram $[\pm^N|\pm^N]:=\bigsqcup_{\lambda\in\Bits{N}}[\lambda|\pm^N]$, we define the ``horizontal halving" $\Gamma$ (i.e., halving the color $\Bits{N}$ of each $\lambda$-fragment) and the ``vertical halving" $/\sim$ (i.e., halving the number of $\lambda$-fragments), and by repeating these steps, ``transport" the rightmost DAG $[+^N|-^N]$ to the leftmost isomorphic DAG $[-^N|+^N]$.
This is a recursive procedure, described by simple symbolic operations.
In addition, the same relational formulas for the characters of DAGs can be used in each recursive step, and the repeated use of them provides a character formula of $[+^N|-^N]$ in terms of the characters of $[\lambda|\pm^N]$.
This is our ``bosonic formula" for $\hat{Z}$-invariants of $(N+2)$-leg star graphs (or corresponding Seifert $3$-manifolds).
In fact, if we assign to each connected component of $[\lambda|\pm^N]$ an appropriate power of $q$ (i.e., ``character of the Fock module"), this almost gives the $\hat{Z}$-invariants, and in the known cases $N=1,2$, it agrees perfectly with the characters of the corresponding simple modules of the singlet Virasoro VOAs.
As mentioned in the beginning of this paper, $\hat{Z}$-invariant itself is a very rich subject with much background.
Nevertheless, the ``bosonic formula" is derived from such an elementary combinatorics.
It suggests that certain aspects of the complexity of the LCFTs and $\hat{Z}$-invariants can be described by the shift system theory (as discussed below, it is our abelian categorification of $\hat{Z}$-invariants) and its Grothendieck group, respectively.

To consider recursive shift systems (in particular, $H^0(\SL\times_B-)$ and the quotients), in Section \ref{section: Lie algebraic aspects} we regard the DAGs in Section \ref{section: combinatorics} as annotated Loewy diagrams in an abelian category.
Our aim is to consider $[\pm^N|\pm^N]$ as a recursive shift system by corresponding $\Gamma$ and $/\sim$ above to $H^0(\SL\times_B-)$ and the quotients in the recursive shift system, respectively.
However, as the name ``fragment" suggests, this is impossible as long as we consider purely discrete and combinatorial situations such as Section \ref{section: combinatorics}.
Because, except for the cases of $N=0,1$, there is no $B$-action commuting with the edges of $[\lambda|\pm^N]$ due to the complicated graph structure.
To solve this problem, as mentioned above, we consider ``pushing" the complexity of the graph to the internal structure of the multiplicity spaces of weight vectors by ``\textit{defragmenting}" $[\pm^N|\pm^N]$.
Here, the degree of freedom by \textit{linear combinations} in the algebraic setting plays an essential role.
That is, when multiple isomorphic simple objects are included, there is an ambiguity regarding how to embed subobjects (in other words, when multiple nodes with the same color are included, we can consider their ``linear combination"), which allows us to deform the DAGs.
\footnote{Our story of ``simplifying computation by considering good linear combinations of bits" is reminiscent of the famous slogan of \textit{quantum computation}. However, the author is a layman and will not go further in this paper.}
Unfortunately, the deformation of $[\pm^N|\pm^N]$ such that it consists a recursive shift system is not yet given in this paper (see Remark \ref{remark: how to construct the B-actions?}).
However, conversely, if a deformation of $[\pm^N|\pm^N]$ consists a recursive shift system, then it is a ``defragmentation" (Definition \ref{def: defragmentation}), i.e., we can continue to regard the multiplicity spaces of weight vectors as hypercube graphs at each recursive steps, as described above (see Corollary \ref{recursive shift system is defragmentation}).

From the above discussion, we can expect an abelian categorization of the $\hat{Z}$-invariants for Seifert $3$-manifolds with $(N+2)$-exceptional fibers (Definition \ref{definition: hypercubic categorification}).
Namely, such a category itself has the nested structure derived from the hypercubic graphs and the mechanism of (de)fragmentations across the hierarchy of the nesting.
As described in \eqref{quotient of the nested FT constructions}, in a recursive shift system, the FT construction must be applied sequentially to each bit (or variable) $\lambda_1,\dots,\lambda_N$, which is  impossible for the ``multivariable" $\lambda$-fragments $[\lambda|\pm^N]$. 
However, by ``univariating" $\bigoplus[\lambda|\pm^N]$ by the defragmentation expressed by the integral sign (Definition \ref{def: defragmentation}), the shift system theory can be applied repeatedly.
On the other hand, if we consider such categories to be the module categories of the corresponding hypothetical LCFTs, comparing our DAGs with those for the known cases $N=1,2$, we can get a glimpse of the correspondence between our DAGs and the hypothetical LCFTs for general $N$ in advance (Section \ref{subsection: glimpse}).  
This paper is only the first step in the quest, but I believe there is a fruitful world in this direction.

\noindent\textbf{Acknowledgements:}
\noindent
First, I would like to express my greatest gratitude to \textit{Sergei Gukov} for his patient interest in the idea of this paper from two years ago and his belief in my potential.
I also thank \textit{Tomoyuki Arakawa}, \textit{Naoki Genra} and \textit{Nicolai Reshetikhin} for their interest and warm encouragement.
Finally, I would like to thank \textit{Hikami Kazuhiro} and \textit{Toshiki Matsusaka}.
We spent less than half a year at Kyushu University, but this work might not have been written without the motivation to understand, express, and explore in my own way the fertile and mysterious $q$-series world that you both gave me a glimpse of.
I am supported by Caltech-Tsinghua joint postdoc Fellowships and by Beijing Natural Science Foundation grant IS24011.

\section{A recursive combinatorial derivation  of $\hat{Z}$-invariants}
\label{section: combinatorics}
\subsection{Directed acyclic graphs}
Note that we will introduce a variety of terms and concepts, some of which are not common and were coined by the author for this paper.
A \textit{directed acyclic graph (DAG)} is a directed graph with no directed cycle.
We use the notation $\operatorname{Node}(\mathcal{Q})$ and $\operatorname{Edge}(\mathcal{Q})\subseteq\operatorname{Node}(\mathcal{Q})\times\operatorname{Node}(\mathcal{Q})$ for the set of nodes and edges of a DAG $\mathcal{Q}$, respectively.
Note that the reachability relation of $\mathcal{Q}$ defines a partial order on $\operatorname{Node}(\mathcal{Q})$.
For $v\in\operatorname{Node}(\mathcal{Q})$, we often write $v\in\mathcal{Q}$ instead of $v\in\operatorname{Node}(\mathcal{Q})$.
For a DAG $\mathcal{Q}$, we sometimes use the symbols ${}_+\mathcal{Q}$ and ${}_-\mathcal{Q}$ for the original DAG $\mathcal{Q}$ and the DAG obtained by reversing the directions of all edges of $\mathcal{Q}$, respectively.
Also, if we can use both $\mathcal{Q}={}_+\mathcal{Q}$ and ${}_{-}\mathcal{Q}$ in some discussion, the symbol ${}_{\pm}\mathcal{Q}$ or ${}_{\mp}\mathcal{Q}$ is used for convenience.
For DAGs $\mathcal{Q},\mathcal{Q}'$ we call $\mathcal{Q}$ is a \textit{subgraph} (resp. \textit{induced subgraph}) of $\mathcal{Q}'$ if $\operatorname{Node}(\mathcal{Q})\subseteq\operatorname{Node}(\mathcal{Q}')$ and $\operatorname{Edge}(\mathcal{Q})\subseteq\operatorname{Edge}(\mathcal{Q}')$ (resp. and in addition, $\operatorname{Edge}(\mathcal{Q'})=\operatorname{Edge}(\mathcal{Q})\cap(\operatorname{Node}(\mathcal{Q}')\times \operatorname{Node}(\mathcal{Q}'))$).
For a DAG $\mathcal{Q}$ and a set $C$, we call $\mathcal{Q}$ is \textit{$C$-colored}  if we have a coloring map $c\colon\operatorname{Node}(\mathcal{Q})\rightarrow C$.
A $C$-colored DAG $\mathcal{Q}$ is \textit{$C$-labeled} if the coloring map is injective.
For simplicity, we often omit these maps and simply say as ``$v$ is labeled/colored by $x$", or  include the label in the symbol representing $v$ (e.g., $v=:v(x)$ for the inverse image of  $x\in C$), or more simply, regard $x$ itself as the corresponding node.
For $C$-colored DAGs $\mathcal{Q}_1$ and $\mathcal{Q}_2$, denote  $\mathcal{Q}_1\simeq\mathcal{Q}_2$ if 
\begin{align*}
((v,w)\in\operatorname{Edge}(\mathcal{Q}_1)\iff(f(v),f(w))\in\operatorname{Edge}(\mathcal{Q}_2))
\ 
\wedge
\ 
(c_1(v)=c_1(w)
\Longrightarrow
c_2(f(v))=c_2(f(w)))
\end{align*}
for $v,w\in\mathcal{Q}_1$ and the coloring maps $c_i\colon\mathcal{Q}_i\rightarrow C$ ($i=1,2$).
If not only $\mathcal{Q}_1\simeq\mathcal{Q}_2$ but also $c_1=c_2\circ f$, then we write as $\mathcal{Q}_1\cong\mathcal{Q}_2$.
For DAGs $\mathcal{Q}_1,\mathcal{Q}_2$, the \textit{Cartesian product} (or \textit{box product})
$\mathcal{Q}_1\square\mathcal{Q}_2$ is defined by
\begin{align}
\label{Cartesian product of DAGs}
\begin{aligned}
&\operatorname{Node}(\mathcal{Q}_1\square\mathcal{Q}_2)
=
\operatorname{Node}(\mathcal{Q}_1)
\times
\operatorname{Node}(\mathcal{Q}_2)
=
\{
(v_1,v_2)
\ |\ 
v_1\in\mathcal{Q}_1,
\ 
v_2\in\mathcal{Q}_2
\},
\\
&\operatorname{Edge}(\mathcal{Q}_1\square\mathcal{Q}_2)
=\left\{
((v_1,v_2),(v'_1,v'_2))
\ \middle| \ 
\begin{aligned}
&(v_1=v'_1\ \land\ (v_2,v_2')\in\operatorname{Edge}(\mathcal{Q}_2))
\\
\lor
&(v_2=v'_2\ \land\ (v_1, v_1')\in\operatorname{Edge}(\mathcal{Q}_1))
\end{aligned}
\right\}.
\end{aligned}
\end{align}

In this paper, all DAGs are  
colored (but in general, not labeled) by 
\begin{align*}
C_m:=\Bits{m}\times\Z
\end{align*}
for some $m\in\Z_{\geq 0}$ (and in most cases, the second term is $\Z_{\geq 0}$).
When a node $v$ is colored by $(b,d)\in C_m$, we call $b$ and $d$ the \textit{bits} (i.e., \textit{binary digits}) and \textit{vertical position} (or \textit{depth}) of $v$, and sometimes denote $b(v)$ and $d(v)$,
respectively (and denote $c(v):=(b(v),d(v))$).
We will later also consider the disconnected union $\mathcal{Q}=\coprod_{h}\mathcal{Q}^h$ such that each $\mathcal{Q}^h$ has the parameter called \textit{horizontal position} (or \textit{Cartan weight}) $h$.
We sometimes use the notation $h(v)=h$ for $v\in\mathcal{Q}^h$, but this parameter is treated as a component of the label, not a color.
For a DAG $\mathcal{Q}$ and $h\in\Z$, $(b,d)\in C_m$, denote 
$\mathcal{Q}^h$
(resp. $\mathcal{Q}_b$, $\mathcal{Q}_d$, $\mathcal{Q}_{b,d}$) the induced subgraphs of $\mathcal{Q}$ with horizontal position $h$ (resp. bits $b$, depth $d$, and $\mathcal{Q}_b\cap\mathcal{Q}_d$).
For $h\in\Z$ and $d\in\Z$, denote $\mathcal{Q}^{[h]}$ and $\mathcal{Q}_{[d]}$ the \textit{horizontal $h$-shift} and the \textit{vertical $d$-shift} of $\mathcal{Q}$, i.e. they are isomorphic to $\mathcal{Q}$ by the identity map, but define
$\mathcal{Q}^{h'}\simeq(\mathcal{Q}^{[h]})^{h+h'}=:\mathcal{Q}^{[h]+h'}$ and $\mathcal{Q}_{d'}\simeq(\mathcal{Q}_{[d]})_{d+d'}=:\mathcal{Q}_{[d]+d'}$
by restriction, respectively.
More generally, for $I\subseteq\Z$, define
$\mathcal{Q}^{[I]}:=\sqcup_{h\in I}\mathcal{Q}^{[h]}$ and $\mathcal{Q}_{[I]}:=\sqcup_{d\in I}\mathcal{Q}_{[d]}$, respectively.
For $\nu\in\Bits{m}$, denote $\mathcal{Q}_{[\nu]}$ the DAG such that $\mathcal{Q}\simeq\mathcal{Q}_{[\nu]}$ by the identity map but the coloring $(b(v),d(v))$ is changed to $(b(v)\ast\nu,d(v)+|\nu|_{-})$ for any $v\in\mathcal{Q}$, where $|\nu|_{-}:=|\{i\ |\ \nu_i=-\}|$. 
By abuse of notation, for $I\subseteq\Z$ and $\nu\in\Bits{m}$, we sometimes write $\mathcal{Q}_{[\nu+I]}$ for $(\mathcal{Q}_{[\nu]})_{[I]}=(\mathcal{Q}_{[I]})_{[\nu]}$.

\subsection{Fragmentation}
\label{subsection: hypercubization}
Let us introduce some terms related to bits.
For $m\in \Z_{\geq 0}$, set
$\Bits{m}=\{+,-\}^{m}$.
For $\lambda\in\{+,-,\pm\}^m$, we write 
$\lambda=\lambda_1\cdots\lambda_m$.
Also, $\lambda_I=\lambda_{i_1\dots i_n}:=\lambda_{i_1}\cdots\lambda_{i_n}\in\Bits{I}$ for $I=(i_1,\dots,i_n)$.
For $x\in\{+,-,\pm\}$, $x^i:=x\cdots x$ of length $i$.
We sometimes regard an element of $\{+,-,\pm\}^m$ as a subset in $\Bits{m}$ by
\begin{align}
\label{regard qbits as subsets}
\{+,-,\pm\}^m
\ni
\hat{\delta}
\mapsto
\{\epsilon
\ |\ 
\hat{\lambda}_i\in\{+,-\}
\Rightarrow
\mu_i=\hat{\lambda}_i
\}
\subseteq
\Bits{m}.
\end{align}
We identify $\Bits{m}$ with the direct product of $m$-copies of the symmetric group $\mathfrak{S}_2$ by 
\begin{align*}
+\mapsto\operatorname{id},
\ 
-\mapsto\sigma_1,
\ \ 
\lambda\ast\lambda'
=
(\lambda\ast\lambda')_1\cdots(\lambda\ast\lambda')_m
:=
(\lambda_1\ast\lambda_1')\cdots(\lambda_m\ast\lambda_m')\in\Bits{m},
\end{align*}
where $\ast$ in the RHS is the product in $\mathfrak{S}_2$.
For $\lambda\in\Bits{m}$, denote $|\lambda|_-:=|\{i\ |\ \lambda_i=-\}|$.
Note that for $\lambda,\mu\in\Bits{m}$,
\begin{align}
\label{inequarity used in tildeV}
|\lambda|_{-}+|\mu|_{-}-|\lambda\ast\mu|_{-}
=2|\{i\ |\ (\lambda_i,\mu_i)=(-,-)\}|\geq 0.
\end{align}
Also, if $\lambda\in\Bits{1}$, then $\lambda 1\in\{\pm1\}$.
For $\lambda,\mu,\mu'\in\Bits{m}$, the \textit{partial order} $\leq_\lambda$ on $\Bits{m}$ is defined by
\begin{align}
\label{partial ordering to define the quantization}
\mu\leq_\lambda\mu'
\iff
\forall i,\ 
\operatorname{gap}(\lambda_i\ast\mu_i)
\leq
\operatorname{gap}(\lambda_i\ast\mu_i').
\iff
\forall i, \ 
((\lambda\ast\mu)_i=+\Rightarrow(\lambda\ast\mu')_i=+).
\end{align}
In particular, $\lambda$ is the maximum element with respect to $\leq_\lambda$.
When $\lambda=+^m$, we often write as $\leq$  instead of $\leq_{+^m}$.
For $\mu\leq_\lambda\mu'$, we also use the letter 
\begin{align*}
|\mu'-_{\lambda}\mu|
:=
\tfrac{1}{2}\sum_{1\leq i\leq m}
(\lambda_i\ast\mu'_i)1
-
(\lambda_i\ast\mu_i)1
=
\{
i\ |\ (\lambda\ast\mu)_i=-\ \land\ (\lambda\ast\mu')_i=+
\}
\in\Z_{\geq 0}.
\end{align*}
In particular, $|\mu-_{\mu}\lambda|=|\lambda\ast\mu|_{-}$.
Finally, for $D\subseteq\Bits{m}$, denote $\bar{D}:=\{\bar\lambda:=-^m\ast\lambda\ |\ \lambda\in D\}$.

The \textit{$m$-cube DAG} is a DAG formed from the nodes and edges of an $m$-dimensional hypercube. 
The following special hypercube DAGs play a fundamental role in this paper.
\begin{definition}
\label{definition: hypercubes}
For $m\in\Z_{\geq 0}$, the $C_m$-colored (but not labeled) $2m$-cube DAG $(\pm^m|\pm^m)$ is defined by
\begin{align}
\label{definition: 2m-DAHG}
\begin{aligned}
&\operatorname{Node}((\pm^m|\pm^m))
=\{
(\lambda|\mu)
\ |\ 
(\lambda,\mu)\in \Bits{m}\times\Bits{m}=\Bits{2m}
\},\\
&\operatorname{Edge}((\pm^m|\pm^m))
=
\left\{
((\lambda|\mu),(\lambda'|\mu'))
\ \middle|\ 
\begin{aligned}
&(|\lambda'-_{+^m}\lambda|=1
\ \land\ 
\mu=\mu')
\\
\lor
&(|\mu-_{+^m}\mu'|=1
\ \land\ 
\lambda=\lambda')
\end{aligned}
\right\}
\\
&b(\lambda|\mu)
:=\lambda\ast\mu,
\ \ 
d(\lambda|\mu)
:=|\lambda\mu|_{-}.
\end{aligned}
\end{align}
More generally, for 
$D,E\subseteq\Bits{m}$
define the induced subgraph $(D|E)$ of $(\pm^m|\pm^m)$ by 
\begin{align*}
\operatorname{Node}(D|E)
=
\{
(\lambda|\mu)
\ |\ 
\lambda\in D,\ \mu\in E
\}.
\end{align*}
This notation is compatible with that of nodes above.
We use this notation with the correspondence \eqref{regard qbits as subsets}.
In particular, for $\lambda,\nu\in\Bits{m}$ and a $C_a$-colored DAG $\mathcal{Q}$, we call $\mathcal{Q}\square(\lambda|\pm^m)_{[\nu]}$ and $\mathcal{Q}\square(\pm^m|\bar\lambda)_{[\nu]}$ the \textit{left $\lambda\ast\nu$-segment} 
and the \textit{right $\overline{\lambda\ast\nu}$-segment} of $\mathcal{Q}\square(\pm^m|\pm^m)_{[\nu]}$, respectively. 
Also we have ${}_{\pm}(D|E)\cong{}_{\mp}(E|D)$, but not equal.
\end{definition}
By definition of $b(\lambda|\mu)$ in \eqref{definition: 2m-DAHG},
$(\lambda|\pm^m)_{[\nu]}$ and $(\pm^m|\lambda)_{[\nu]}$ are the \textit{$C_m$-labeled $m$-cube DAGs}.
Note that
\begin{align}
\label{the relation between standard case and epsilon twisted case}
\begin{aligned}
&(\lambda|\pm^m)_{[\nu]}
\cong
(\lambda\ast\nu|\pm^m)_{[|\lambda|_{-}+|\nu|_{-}-|\lambda\ast\nu|_{-}]}
\overset{\eqref{inequarity used in tildeV}}{=}
(\lambda\ast\nu|\pm^m)_{2|(\lambda_i,\nu_i)=(-,-)|},
\\
&(\pm^m|\bar\lambda)_{[\nu]}
\cong
(\pm^m|\overline{\lambda\ast\nu})_{[-|\lambda|_{-}+|\nu|_{-}+|\lambda\ast\nu|_{-}]}
\overset{\eqref{inequarity used in tildeV}}{=}
(\pm^m|\overline{\lambda\ast\nu})_{2|(\lambda_i,\nu_i)=(+,-)|}.
\end{aligned}
\end{align}
In particular, $(+^m|\pm^m)_{[\lambda]}=(\lambda|\pm^m)$ and $(\pm^m|-^m)_{[\lambda]}
=(\pm^m|\bar\lambda)_{2|\lambda|_{-}}$, respectively.
Hypercube graphs are one of the simplest of the \textit{``recursive" (or ``nested") structures}, i.e. the $(m+1)$-cube is obtained by shifting the $m$-cube. 
More generally, the following lemma holds (the proof is straightforward).
\begin{lemma}
\label{Lemma: Cartesian product of left/right hypercubes}
For $\lambda,\mu\in\Bits{m}$, let $\hat{\lambda},\hat{\lambda}',\hat{\mu},\hat{\mu}'\in\{+,-,\pm\}^{m}$ be elements such that for each $1\leq i\leq m$, 
\begin{align*}
\lambda_i=+\iff\hat{\lambda}_i=\hat{\lambda}'_i=+,
\ \ 
\mu_i=-\iff\hat{\mu}_i=\hat{\mu}'_i=-,
\ \ 
\hat{\lambda}_i=\pm
\Rightarrow
\hat{\lambda}'_i\not=\pm,
\ \ 
\hat{\mu}_i=\pm
\Rightarrow
\hat{\mu}'_i\not=\pm,
\end{align*}
Then we have the isomorphism
\begin{align*}
(\hat{\lambda}|\hat{\mu})\square(\hat{\lambda}'|\hat{\mu}')
\simeq 
(\hat{\lambda}'|\hat{\mu}')
\square
(\hat{\lambda}|\hat{\mu})
\simeq
(\hat{\lambda}\cup\hat{\lambda}'|\hat{\mu}\cup\hat{\mu}')_{[d(\lambda,\mu)]},
\end{align*}
where $\hat{\lambda}\cup\hat{\lambda}'$ and $\hat{\mu}\cup\hat{\mu}'$ defines the corresponding elements in $\{+,-,\pm\}^{m+m'}$.
In other words, for $m_1=|\{i|\hat\lambda_i=\pm\}|$ and $m_2=|\{i|\hat\mu_i=\pm\}|$, the two independent $m_1$- and $m_2$-DAHGs $(\hat{\lambda}|\mu)$ and $(\lambda|\hat\mu)$, which share the starting node $(\lambda|\mu)$, span the (vertical $[d(\lambda|\mu)]$-shifted) $(m_1+m_2)$-DAHG $(\hat\lambda|\hat\mu)$ in $(\pm^m|\pm^m)$ by the Cartesian product.
In particular, by identifying $(\lambda|\pm^m)$ and $(\lambda'|\pm^{m'})$ with $(\lambda\lambda'|\pm^m+^{m'})_{[-|\lambda'|_-]}$ and $(\lambda\lambda'|+^m\pm^{m'})_{[-|\lambda|_-]}$, respectively, we have
$(\lambda|\pm^m)\square(\lambda'|\pm^{m'})
\simeq
(\lambda\lambda'|\pm^{m+m'})$.
More generally, for $(\nu,\nu')\in\Bits{m}\times\Bits{m'}$, we have $(\lambda|\pm^m)_{[\nu]}\square(\lambda'|\pm^{m'})_{[\nu']}
\simeq
(\lambda\lambda'|\pm^{m+m'})_{[\nu\nu']}$,
where note that $\nu\nu'=(\nu +^{m'})\ast(+^m\nu')$.
\end{lemma}
In the following, if there is no risk of confusion, we will write $\mathcal{Q}(D|E)_{[\nu]}$ instead of $\mathcal{Q}\square(D|E)_{[\nu]}$ and use the above identification.
For convenience, let us introduce the following concept necessary to our discussion.
\begin{definition}
\label{definition: slim}
For $a\in\Z_{\geq 0}$ and a $C_a$-colored DAG $\mathcal{Q}$, we call $\mathcal{Q}$ is \textit{slim} if $\mathcal{Q}$ is 
$C_a$-labeled and satisfies
\begin{align}
\label{positive even vertical shiftable}
(b,d)\in\mathcal{Q}
\Rightarrow
(b,d+2)\in\mathcal{Q}.
\end{align}
We call $\mathcal{Q}$ is \textit{fat} if it is not slim.
For example, if $\mathcal{Q}$ has the form $\mathcal{Q}'(\pm^m|\pm^m)_{[\nu]}$ for some $m\geq 1$, then $\mathcal{Q}$ is fat.
\end{definition}
Let $\mathcal{Q}$ be a slim DAG.
For $m\in\Z_{\geq 0}$ and $\lambda,\nu\in\Bits{m}$, the Cartesian product $\mathcal{Q}(\lambda|\pm^m)_{[\nu]}$ is labeled by
\begin{align*}
(bb(\lambda|\mu)_{[\nu]},d+d(\lambda|\mu)_{[\nu]})
:=
((b,d),(\lambda|\mu)_{[\nu]})
\in
C_{a+m}
\ \ 
(b\in\Bits{a},\ \mu\in\Bits{m}).
\end{align*}
On the other hand, since
\begin{align*}
(bb(\lambda|\mu)_{[\nu]},d+d(\lambda|\mu)_{[\nu]})
=
(bb(\bar\mu|\bar\lambda)_{[\nu]},
(d+2d(\lambda|\mu)_{[\nu]})
+
(d(\bar\mu|\bar\lambda)_{[\nu]}-2m)),
\end{align*}
we have the inclusion map
$\operatorname{Node}(\mathcal{Q}(\lambda|\pm^m))
\hookrightarrow
\operatorname{Node}
(\mathcal{Q}(\pm^m|\bar\lambda)_{[\nu-2m]})$ defined by
\begin{align}
\label{identification of left-right labelings}
\mathcal{Q}(\lambda|\pm^m)_{[\nu]}
\ni
((b,d),
(\lambda|\mu)_{[\nu]})
=
((b,d+2(\lambda|\mu)_{[\nu]}),
(\bar\mu|\bar\lambda)_{[\nu-2m]})
\in
\mathcal{Q}((\pm^m|\bar\lambda)_{[\nu-2m]}),
\end{align}
under the similar labeling of $\mathcal{Q}((\pm^m|\bar\lambda)_{[\nu-2m]})$
(note that we use the condition \eqref{positive even vertical shiftable} here).
In the same manner, we also have the inclusion $\operatorname{Node}
(\mathcal{Q}(\pm^m|\bar\lambda)_{[\nu]})
\hookrightarrow
\operatorname{Node}
((\mathcal{Q}(\lambda|\pm^m)_{[\nu-2m]}))$ 
by
\begin{align}
\label{identification of right-left labelings}
\mathcal{Q}(\pm^m|\bar\lambda)_{[\nu]}
\ni
((b,d),
(\bar\mu,\bar\lambda)_{[\nu]}
=
((b,d+2(\bar\mu|\bar\lambda)_{[\nu]}),
((\lambda|\mu)_{[\nu-2m]}))
\in
\mathcal{Q}((\lambda|\pm^m)_{[\nu-2m]}).
\end{align}

\begin{definition}
\label{definition: left/right delta cubizations}
Let $m\in\Z_{\geq 0}$, $\lambda,\nu\in\Bits{m}$, and $\mathcal{Q}$ be a slim DAG.
\begin{enumerate}
\item 
The \textit{left $(\lambda\ast\nu)$-fragment $\mathcal{Q}[\lambda|\pm^m)_{[\nu]}$ of $\mathcal{Q}(\pm^m|\pm^m)_{[\nu]}$}
is defined by
\begin{align*}
&\operatorname{Node}(\mathcal{Q}[\lambda|\pm^m)_{[\nu]})
:=
\operatorname{Node}(\mathcal{Q}(\lambda|\pm^m)_{[\nu]})
\underset{\eqref{identification of left-right labelings}}{\hookrightarrow}
\operatorname{Node}(\mathcal{Q}(\pm^m|\bar\lambda)_{[\nu-2m]}),\\
&(v,v')
\in
\operatorname{Edge}(\mathcal{Q}[\lambda|\pm^m)_{[\nu]})
\iff
(v,v')
\in
\operatorname{Edge}(\mathcal{Q}(\lambda|\pm^m)_{[\nu]})
\cup
\operatorname{Edge}
(\mathcal{Q}(\pm^m|\bar\lambda)_{[\nu-2m]}).
\end{align*}
\item 
The \textit{right $(\overline{\lambda\ast\nu})$-fragment $\mathcal{Q}(\pm^m|\bar\lambda]_{[\nu]}$ of $\mathcal{Q}(\pm^m|\pm^m)_{[\nu]}$} is defined by
\begin{align*}
&\operatorname{Node}(\mathcal{Q}(\pm^m|\bar\lambda])_{[\nu]})
:=
\operatorname{Node}(\mathcal{Q}(\pm^m|\bar\lambda)_{[\nu]})
\underset{\eqref{identification of right-left labelings}}{\hookrightarrow}
\operatorname{Node}(\mathcal{Q}(\lambda|\pm^m)_{[\nu-2m]}),\\
&(v,v')
\in
\operatorname{Edge}(\mathcal{Q}(\pm^m|\bar\lambda])_{[\nu]}
\iff
(v,v')
\in
\operatorname{Edge}(\mathcal{Q}(\pm^m|\bar\lambda)_{[\nu]})
\cup
\operatorname{Edge}
(\mathcal{Q}(\lambda|\pm^m)_{[\nu-2m]}).
\end{align*}
\end{enumerate}
We also use the well-defined notation $\mathcal{Q}{}_{\pm}[\lambda|\pm^m)\simeq\mathcal{Q}{}_{\mp}(\pm^m|\lambda]$.
Note that by \eqref{the relation between standard case and epsilon twisted case}, we have
\begin{align*}
\mathcal{Q}[\lambda|\pm^m)_{[\nu]}
\cong
\mathcal{Q}[\lambda\ast\nu|\pm^m)_{[|\lambda|_{-}+|\nu|_{-}-|\lambda\ast\nu|_{-}]}
\ \ 
\text{(resp. $\mathcal{Q}(\pm^m|\bar\lambda]_{[\nu]}
\cong
\mathcal{Q}(\pm^m|\overline{\lambda\ast\nu}]_{[-|\lambda|_{-}+|\nu|_{-}+|\lambda\ast\nu|_{-}]}$)}.
\end{align*}
For a (not necessary slim) subgraph $\mathcal{Q'}\subseteq\mathcal{Q}$ and $E\subseteq\Bits{m}$,
denote
$\mathcal{Q}'[\lambda|E)_{[\nu]}$ the induced subgraph of $\mathcal{Q}[\lambda|\pm^m)_{[\nu]}$ defined by
$\operatorname{Node}(\mathcal{Q}'[\lambda|E)_{[\nu]}
)
=
\operatorname{Node}(\mathcal{Q}'\square(\lambda|E)_{[\nu]})$
($\mathcal{Q}'(E|\bar\lambda]_{[\nu]}\subseteq\mathcal{Q}(\pm^m|\bar\lambda]_{[\nu]}$ is defined in the same manner).
\end{definition}
\begin{remark}
\label{remark: segment and fragment}
The \textit{$(\lambda\ast\nu)$-segments} $\mathcal{Q}(\lambda|\pm^m)_{[\nu]}$ are subgraphs of $\mathcal{Q}(\pm^m|\pm^m)_{[\nu]}$, wheareas the \textit{left $(\lambda\ast\nu)$-fragments} $\mathcal{Q}[\lambda|\pm^m)_{[\nu]}$ are \textit{not}, because edges are not included (the cases $\mathcal{Q}(\pm^m|\bar\lambda)_{[\nu]}$ and $\mathcal{Q}(\pm^m|\bar\lambda]_{[\nu]}$ are the same).
This is why we distinguish between the slightly different terms ``segment" and ``fragment", or between $(\ |\ )$ and $[\ |\ ),(\ |\ ]$.
Put another way, $(\ |\ )$ and $[\ |\ ),(\ |\ ]$ are compatible in $\operatorname{Node}(\ )$, but not at the graph level.
\end{remark}

The following \textit{composition} of slim cubizations follows from that of hypercubes in Lemma \ref{Lemma: Cartesian product of left/right hypercubes}.
\begin{lemma}
\label{Lemma: associativity of lert/right cubizations}
For $m,m'\in\Z_{\geq 0}$, $(\lambda,\lambda'), (\nu,\nu')\in\Bits{m}\times\Bits{m'}$, 
and
a slim DAG $\mathcal{Q}$,
we have
\begin{align*}
\mathcal{Q}[\lambda|\pm^m)_{[\nu]}[\lambda'|\pm^{m'})_{[\nu']}
\cong
\mathcal{Q}[\lambda\lambda'|\pm^{m+m'})_{[\nu\nu']},
\ \ 
\mathcal{Q}(\pm^m|\lambda]_{[\nu]}(\pm^{m'}|\lambda']_{[\nu']}
\cong
\mathcal{Q}(\pm^{m+m'}|\lambda\lambda']_{[\nu\nu']}.
\end{align*}
\end{lemma}
Later, we will consider the following (fat) DAG formed by aligning the left/right $\lambda$-fragments.
\begin{definition}
\label{definition: fragmented fat cubization}
Let $a,m\in\Z_{\geq 0}$, $\nu\in\Bits{m}$, and $\mathcal{Q}$ be a slim $C_a$-labeled DAG.
The \textit{left/right-fragmented $\mathcal{Q}(\pm^m|\pm^m)_{[\nu]}$} are the fat $C_{a+m}$-labeled DAGs $\mathcal{Q}[\pm^m|\pm^m)_{[\nu]}$ and $\mathcal{Q}(\pm^m|\pm^m]_{[\nu]}$ defined by
\begin{align*}
\mathcal{Q}[\pm^m|\pm^m)_{[\nu]}
:=
\sqcup_{\lambda\in\Bits{m}}
\mathcal{Q}[\lambda|\pm^m)_{[\nu]},
\ \ 
\mathcal{Q}(\pm^m|\pm^m]_{[\nu]}
:=
\sqcup_{\lambda\in\Bits{m}}
\mathcal{Q}(\pm^m|\bar{\lambda}]_{[\nu]},
\end{align*}
respectively.
As the name suggests, we have
\begin{align*}
\operatorname{Node}(\mathcal{Q}[\pm^m|\pm^m)_{[\nu]})
=
\operatorname{Node}(\mathcal{Q}(\pm^m|\pm^m]_{[\nu]})
=
\operatorname{Node}(\mathcal{Q}(\pm^{m}|\pm^{m})_{[\nu]}).
\end{align*}
\end{definition}

Let us apply all the above operations to the simplest case, i.e., set the slim $C_0$-labeled DAG
\begin{align}
\label{the simplest tower}
\mathcal{Q}=[\ |\ )=(\ |\ ],
\ \ 
\operatorname{Node}[\ |\ )
:=
\{(\phi,d)\ |\ d\in 2\Z_{\geq 0}\},
\ \ 
\operatorname{Edge}[\ |\ )
:=
\phi.
\end{align}
For $m\in\Z_{\geq 0}$ and  $\lambda,\nu\in\Bits{m}$, the slim $C_m$-labeled DAGs
\begin{align*}
[\lambda|\pm^m)_{[\nu]}
:=
[\ |\ )[\lambda|\pm^m)_{[\nu]},
\ \ 
(\pm^m|\lambda]_{[\nu]}
:=
(\ |\ ](\pm^m|\lambda]_{[\nu]}
\end{align*}
are defined
by Definition \ref{definition: left/right delta cubizations}.
By Lemma \ref{Lemma: associativity of lert/right cubizations}, this notation is well-defined, i.e. for $m'\in\Z_{\geq 0}$ and $\lambda',\nu'\in\Bits{m'}$,
\begin{align*}
[\lambda\lambda'|\pm^{m+m'})_{[\nu\nu']}
:=
[\lambda|\pm^m)_{[\nu]}
[\lambda'|\pm^{m'})_{[\nu']},
\ \ 
(\pm^{m+m'}|\lambda\lambda']_{[\nu\nu']}
:=
(\pm^m|\lambda]_{[\nu]}
(\pm^{m'}|\lambda']_{[\nu']}.
\end{align*}
For $D,E\subseteq\Bits{m}$, the DAGs $[D|E)_{[\nu]}$,  $(\bar{E}|\bar{D}]_{[\nu]}$, etc. are defined in the obvious way.

\subsection{Derivation of $\hat{Z}$-invariants}
\label{subsection: derivation of Z-invariants}
Let us define the \textit{character} $\ch\mathcal{Q}$ of a $C_m$-colored DAG $\mathcal{Q}$.
In some sense, it measures how far the $\mathcal{Q}$ is from $C_m$-labeled DAGs, or how much extra information besides $C_m$.
\begin{definition}
For $m\in\Z_{\geq 0}$ and a $C_m$-colored DAG $\mathcal{Q}$, the \textit{character} of $\mathcal{Q}$ is defined by
\begin{align*}
\ch\mathcal{Q}
:=
\sum_{(b,d)\in C_m}
|\operatorname{Node}(\mathcal{Q}_{b,d})|
x_{b,d},
\end{align*}
where $x_{b,d}$ are variants.
For $C_m$-colored DAGs $\mathcal{Q},\mathcal{Q}'$, we consider the \textit{equivalence relation} $\ch\mathcal{Q}\sim\ch\mathcal{Q}'$ by
\begin{align*}
\forall(b,d)\in C_m,\ 
(|\operatorname{Node}(\mathcal{Q}_{b,d})|=0
\iff
|\operatorname{Node}(\mathcal{Q}'_{b,d})|=0).
\end{align*}
Note that if $\mathcal{Q}$ is $C_m$-labeled, then all coefficients are $0$ or $1$.
In particular, if $\ch\mathcal{Q}\sim\ch\mathcal{Q}'$ and both $\mathcal{Q}$ and $\mathcal{Q}'$ are $C_m$-labeled, then $\ch\mathcal{Q}=\ch\mathcal{Q}'$.
Also note that $\ch\mathcal{Q}$ depends only on $\operatorname{Node}(\mathcal{Q})$ and the coloring, not on $\operatorname{Edge}(\mathcal{Q})$.
\end{definition}
Let us fix $N\in\Z_{\geq 1}$ and $\lambda_1\in\Bits{1}$.
Our goal in this subsection is to give a nice \textit{recursive expression} of
\begin{align}
\label{the starting socle nodes}
[\lambda_1|+)[+^{N-1}|-^{N-1})
\end{align}
i.e., the $(\lambda_1-^{N-1})$-component of ${}_{\pm}[\lambda_1|\pm)[+^{N-1}|\pm^{N-1})$,
and the character formula of \eqref{the starting socle nodes} using such an expression.
For $m\in\Z_{\geq 0}$, $\lambda_a\in\Bits{1}$ and a slim DAG $\mathcal{Q}$, let us consider the two operations
\begin{align}
&\Gamma(\mathcal{Q}{}_{\pm}[\lambda_a|\pm)[\pm^m|\pm^m))
:=
\mathcal{Q}[\lambda_a|+)[\pm^m|\pm^m),
\label{take restricted SL2 component from restricted B}
\\
&\mathcal{Q}[\pm^m|\pm^m)\slash{\sim}
:=
\mathcal{Q}[-\pm^{m-1}|\pm\pm^{m-1}).
\label{take restricted B component from restricted SL2}
\end{align}
For convenience, we write $\tilde{\Gamma}(\mathcal{
Q})$ for the composition $\Gamma(\mathcal{Q})\slash{\sim}$.
Then we have the recursive expression
\begin{align}
\label{the recursive expression of the restricted starting socle nodes}
[\lambda_1|+)[+^{N-1}|-^{N-1})
\cong
[\lambda_1|+)[-^{N-1}|+^{N-1})
=
\underbrace{\Gamma(\tilde{\Gamma}(\cdots\tilde{\Gamma}}_{\text{$N$-times}}({}_{\pm}[\lambda_1|\pm)[\pm^{N-1}|\pm^{N-1}))\cdots))
\end{align}
of \eqref{the starting socle nodes}.
On the other hand, the characters on both sides of \eqref{take restricted SL2 component from restricted B}, \eqref{take restricted B component from restricted SL2} are related by the formulas
\begin{align}
&\ch\mathcal{Q}[\lambda_{a}|+)[\pm^{m}|\pm^{m})_{[\nu]}
=
\sum_{n\geq 0}
\sum_{\nu_{a}\in\Bits{1}}
(-1)^{|\nu_{a}|_{-}}
\ch\mathcal{Q}{}_{\pm}[\lambda_{a}|\pm)_{[\nu_a][2n]}[\pm^{m}|\pm^{m})_{[\nu]}\ \ (\nu\in\Bits{m}),
\label{the formula connectiong SL2 to B}
\\
&
\begin{aligned}
\ch\mathcal{Q}{}_{\pm}[\lambda_{a}|\pm)_{[\nu_a]}[\pm^{m}|\pm^{m})_{[\nu]}
&\sim
\ch\mathcal{Q}[\pm|\pm)_{[\lambda_{a}\ast\nu_{a}][2|(\lambda_a,\nu_a)=(-,-)|]}[\pm^{m}|\pm^{m})_{[\nu]}
\\
&=
\ch\mathcal{Q}[\pm^{m+1}|\pm^{m+1})_{[(\lambda_{a}\ast\nu_a)\nu][2|(\lambda_a,\nu_a)=(-,-)|]}.
\end{aligned}
\label{the formula connectiong B to SL2}
\end{align}
Note that by Remark \ref{remark: segment and fragment}, $[\ |\ )$ can be replaced by $(\ |\ ]$ in these formulas.
Applying them to $[-^m|+^m)$, we have
\begin{align}
\label{in fact they are equal}
\begin{aligned}
\ch[-^m|+^m)
&\sim
\sum_{n_1,\dots,n_m\in\Z_{\geq 0}}
\sum_{\nu\in\Bits{m}}
(-1)^{|\nu|_{-}}
\ch[\pm^m|\pm^m)_{[\bar\nu][2(n_1+\cdots+n_m+|\nu|_{-})]}
\\
&\sim
\sum_{n\in\Z_{\geq 0}}\binom{n+m-1}{m-1}
\sum_{\nu\in\Bits{m}}
(-1)^{|\nu|_{-}}
\ch[\bar\nu|\pm^m)_{[2(n+|\nu|_{-})]},
\end{aligned}
\end{align}
where we use $\ch[\pm^m|\pm^m)_{[\nu]}\sim\ch[+^m|\pm^m)_{[\nu]}$ and \eqref{the relation between standard case and epsilon twisted case} in the second equivalence.
Moreover, both sides of \eqref{in fact they are equal} are equal.
To verify this fact, it suffices to check that for any $k\in\Z_{\geq 0}$, the coefficient of $x_{-^m,m+2k}$ in the RHS of \eqref{in fact they are equal} is $1$.
Since $[\bar\nu|\pm^m)_{[2(n+|\nu|_{-})]}$ is $C_m$-labeled and 
$(-^m,m+2k)=(\bar\nu|\nu)_{[2k]}\in[\bar\nu|\pm^m)_{[2(n+|\nu|_{-})]}$
iff $n\leq k-|\nu|_{-}$,
the coefficient of $x_{-^m,m+2k}$ in the RHS of \eqref{in fact they are equal} is given by
\begin{align*}
\sum_{\nu\in\Bits{m}}
(-1)^{|\nu|_{-}}
\sum_{0\leq n\leq k-|\nu_{-}|}
\binom{m+n-1}{m-1}.
\end{align*}
By induction for $k$, to check that this is equal to $1$, it is enough to check that
\begin{align*}
\sum_{0\leq a\leq m}
(-1)^a
\binom{m}{a}\binom{m+k-a}{m-1}=0
\ \ 
(\forall k\in\Z_{\geq 0}).
\end{align*}
It is verified by substituting $p(x)=\tfrac{(x+k)\cdots(x+k-m+2)}{(m-1)!}$ and $x=m$ to the algebraic identity \cite{R}
\begin{align*}
\sum_{0\leq a\leq m}
(-1)^a
\binom{m}{a}p(x-a)=m!a_m
\ \ 
(p(x):=\sum_{0\leq i\leq m}a_ix^i\in\C[x]).
\end{align*}
Therefore, both sides of \eqref{in fact they are equal} are equal.
Finally by replacing $m$ by $N-1$ and applying \eqref{the formula connectiong SL2 to B} again, we have
\begin{align}
\label{bosonic formula of the Z-invariants of Seifert 3-manifolds}
\ch[\lambda_1|+)[+^{N-1}|-^{N-1})
=
\sum_{n\geq 0}\binom{n+N-1}{N-1}
\sum_{\nu\in\Bits{N}}
(-1)^{|\nu|_{-}}
\ch{}_{\pm}[\lambda_1|\pm)[+^{N-1}|\pm^{N-1})_{[\nu_1\overline{\nu_{2\dots N}}][2(n+|\nu_{2\dots N}|_{-})]}.
\end{align}
For convenience, set $\lambda\ast{}_{\pm}[+|\pm)[+^{N-1}|\pm^{N-1}):={}_{\pm}[\lambda_1|\pm)[\lambda_{2\dots N}|\pm^{N-1})_{[-|\lambda|_{-}]}$.
Then \eqref{bosonic formula of the Z-invariants of Seifert 3-manifolds} is expressed as
\begin{align}
\label{bosonic formula of the Z-invariants of Seifert 3-manifolds 2}
\sum_{n\geq 0}\binom{n+N-1}{N-1}
\sum_{\nu\in\Bits{N}}
(-1)^{|\nu|_{-}}
\ch(\lambda_1\ast\nu_1)\overline{\nu_{2\dots N}}\ast{}_{\pm}[+|\pm)[+^{N-1}|\pm^{N-1})_{[2n+N+|\nu|_{-}-|\overline{\lambda_1}|_{-}]}.
\end{align}

Let us explain that \eqref{bosonic formula of the Z-invariants of Seifert 3-manifolds} can be regarded as a ``\textit{bosonic formula}" of the \textit{$\hat{Z}$-invariants of Seifert $3$-manifolds with $(N+2)$-singular fibers} (see also Remark \ref{remark: restriction of scope in the present paper}).
Let $p_1,\dots,p_N\in\Z_{\geq 2}$ are coprime integers and fix $1\leq r_i< p_i$ for each $1\leq i\leq N$.
For $\lambda\in\Bits{N}$ and $d\in\Z_{\geq 0}$, let us define the character $\ch\lambda\ast{}_{\pm}[+|\pm)[+^{N-1}|\pm^{N-1})_{[d]}$ by
\begin{align}
\label{character of general Fock module}
\ch\lambda\ast{}_{\pm}[+|\pm)[+^{N-1}|\pm^{N-1})_{[d]}
:=
\tfrac{1}{\eta(q)}q^{\Delta_{\lambda,d}}
:=
\tfrac{1}{\eta(q)}
q^{\frac{p_1\cdots p_N}{2}(-d+\sum_{1\leq i\leq N}\tfrac{\lambda_ir_i}{p_i})^2}.
\end{align}
In other words, at least in the character level, $\lambda\ast{}_{\pm}[+|\pm)[+^{N-1}|\pm^{N-1})_{[d]}$ is regarded as a \textit{Fock module}.
Then 
\begin{align}
\label{character of general singlet module}
\eqref{bosonic formula of the Z-invariants of Seifert 3-manifolds 2}
=
\frac{1}{\eta(q)}\sum_{n\geq 0}\binom{n+N-1}{N-1}
\sum_{\nu\in\Bits{N}}
(-1)^{|\nu|_{-}}
q^{\tfrac{p_1\cdots p_N}{2}(-2n-N-|\nu|_{-}+|\overline{\lambda_1}|_{-}+\frac{(\lambda_1\ast\nu_1)r_1}{p_1}-\sum_{i=2}^N\frac{\nu_ir_i}{p_i})^2}.
\end{align}
In the cases $N=1,2$, \eqref{character of general Fock module} and \eqref{character of general singlet module} exactly coincides with the characters of the corresponding Fock module and the irreducible module of the \textit{singlet Virasoro algebra} at level $1/p_1$ and $p_2/p_1$, and they almost give the ``\textit{bosonic form}" of the $\hat{Z}$-invariants of the corresponding Seifert $3$-manifolds with $3$- or $4$-exceptional fibers, respectively (e.g., \cite[Section 3,4]{CFGHP}).
For $N\geq 3$, the corresponding LCFT is not yet known at least mathematically, but at the $q$-series level, \eqref{character of general singlet module} almost gives the $\hat{Z}$-invariants of the corresponding Seifert $3$-manifolds with $(N+2)$-exceptional fibers (e.g., \cite[Definition 2.1]{MT}).
In other words, \eqref{character of general singlet module} is expected to be the character of the irreducible module of the \textit{conjectural} LCFT corresponding to the  Seifert $3$-manifold with $(N+2)$-exceptional fibers.
For more detail, see Section \ref{subsection: glimpse} below.

In conclusion, we have already obtained a  \textit{recursive combinatorial computation algorithm} (or an \textit{elementary combinatorial interpretation}) for the bosonic formula of the $\hat{Z}$-invariants for Seifert $3$-manifolds. 
In the remainder of this paper, we will give additional structures necessary to connect this algorithm to the study of the \textit{representation theoretic aspect} of LCFT, i.e., the theory of \textit{shift systems and Feigin-Tipunin constructions} \cite{LS}.

\subsection{Horizontal extension}
\label{subsection: Cartan extension}
From now on, we consider the horizontal position (Cartan weight) of DAGs. 
In this subsection, after introducing the \textit{horizontal extension} of DAGs in Section \ref{subsection: hypercubization}, we will explain how the recursive procedure in Section \ref{subsection: derivation of Z-invariants} is interpreted.
This interpretation plays a fundamental role in Section \ref{section: Lie algebraic aspects}.
\begin{definition}
\label{definition: bilateral DAGs}
We call DAGs with the form
$\mathcal{Q}:=\mathcal{Q}^{\mathrm{left}}\sqcup\mathcal{Q}^{\mathrm{right}}$ \textit{bilateral}.
For a bilateral DAG  $\mathcal{Q}$, the \textit{horizontal extension} ${\mathcal{Q}}_{[2\Z_{\geq 0}]}$ and the \textit{reverse} $\bar{\mathcal{Q}}$ of $\mathcal{Q}$ are the bilateral DAGs defined by 
$({\mathcal{Q}}_{[2\Z_{\geq 0}]})^{\mathrm{left/right}}:=\mathcal{Q}^{\mathrm{left/right}}_{[2\Z_{\geq 0}]}$ and
$\bar{\mathcal{Q}}^{\mathrm{left/right}}:=\mathcal{Q}^{\mathrm{right/left}}$, respectively.
Also, we use the notations
\begin{align*}
\leftop{Q}
:=
\min\{d(v)
\ | \ 
v\in\Left{Q}\},\ \ 
\rightop{Q}
:=
\min\{d(v)\ | \ v\in\Right{Q}\},\ \ 
\gap{Q}
:=
\rightop{Q}-\leftop{Q}.
\end{align*}
By abuse of notation, we call $\mathcal{Q}$ or $\mathcal{Q}_{[2\Z_{\geq 0}]}$ is $C_a$-labeled (resp. slim) if  $\mathcal{Q}^{\mathrm{left}}$ and $\mathcal{Q}^{\mathrm{right}}$ so are.
\end{definition}

\begin{definition}
\label{definition: delta cubizations of bilateral DAGs}
For $m\in\Z_{\geq 0}$, $\lambda,\nu\in\Bits{m}$, and a slim bilateral DAG $\mathcal{Q}=\mathcal{Q}^{\mathrm{left}}\sqcup\mathcal{Q}^{\mathrm{right}}$, the \textit{fragmented $\mathcal{Q}(\pm^m|\pm^m)_{[\nu]}$}
is the fat bilateral DAG $\mathcal{Q}[\pm^m|\pm^m]_{[\nu]}=\mathcal{Q}[\pm^m|\pm^m]_{[\nu]}^{\mathrm{left}}\sqcup\mathcal{Q}[\pm^m|\pm^m]_{[\nu]}^{\mathrm{right}}$ defined by
\begin{align*}
\mathcal{Q}[\pm^m|\pm^m]_{[\nu]}^{\mathrm{left}}
:=
\mathcal{Q}[\pm^m|\pm^m)_{[\nu]},
\ \ 
\mathcal{Q}[\pm^m|\pm^m]_{[\nu]}^{\mathrm{right}}
:=
\mathcal{Q}(\pm^m|\pm^m]_{[\nu]}.
\end{align*}
More generally, for $D,E\subseteq\Bits{m}$, we have the corresponding induced subgraph
\begin{align*}
\mathcal{Q}[D|E]_{[\nu]}
:=
\mathcal{Q}[D|E)_{[\nu]}
\sqcup
\mathcal{Q}(\bar{E}|\bar{D}]_{[\nu]}
\end{align*}
of $\mathcal{Q}[\pm^m|\pm^m]_{[\nu]}$.
In particular, for $\lambda\in\Bits{m}$, we call $\mathcal{Q}[\lambda|\pm^m]_{[\nu]}$ the \textit{$(\lambda\ast\nu)$-fragment of $\mathcal{Q}[\pm^m|\pm^m]_{[\nu]}$}.
In the case where $\mathcal{Q}=[\ |\ ]:=[\ |\ )\sqcup(\ |\ ]$, we will omit $\mathcal{Q}$ as in the last of Section \ref{subsection: hypercubization}.
As in the same manner above, this notation is well-defined and the compositions of the operations also hold by Lemma \ref{Lemma: associativity of lert/right cubizations}.
\end{definition}
Note that under the operations in Definition \ref{definition: delta cubizations of bilateral DAGs}, the gap of a bilateral DAG $\mathcal{Q}$ changes as follows:
\begin{align}
\label{several formulas about gap}
\operatorname{gap}(\mathcal{Q}[D|E]_{[\nu]})
=
\operatorname{gap}(\mathcal{Q})+|\bar{D}|_{-}+|\bar{E}|_{-}-|{D}|_{-}-|{E}|_{-},
\end{align}
In particular, 
$\operatorname{gap}(\mathcal{Q}[\lambda|\pm^m]_{[\nu]})
=\operatorname{gap}(\mathcal{Q})+m-2|\lambda|_{-}$
and
$\operatorname{gap}(\mathcal{Q}[\pm^m|\pm^m]_{[\nu]})=\operatorname{gap}(\mathcal{Q})$.
To consider the bilateral extensions of the DAGs in Section \ref{subsection: derivation of Z-invariants}, the following notations will be convenient.

\begin{definition}
\label{definition: X-type DAGs}
A slim bilateral DAG $\mathcal{Q}_{[2\Z_{\geq 0}]}$ is called \textit{$(B_+,0)$-DAG} 
if 
$\mathcal{Q}
\simeq
[\lambda_1|\pm]$
for some $\lambda_1\in\Bits{1}$.
Note that if $\mathcal{Q}\cong[\mu\lambda_1|\mu'\pm]_{[\nu]}$ for some $n\in\Z_{\geq 0}$, $\mu,\mu'\in\Bits{n}$ and $\nu\in\Bits{n+1}$, then this condition is equivalent to  
$\operatorname{gap}([\mu\lambda_1|\mu'\pm]_{[\nu]})\in\{\pm 1\}$ 
(i.e., $n-|\mu\mu'\lambda_1|_{-}\in\{0,-1\}$).
Also, 
${}_{-}\mathcal{Q}_{[2\Z_{\geq 0}]}$ and the subgraph of ${}_{\pm}\mathcal{Q}_{[2\Z_{\geq 0}]}$ corresponding to $[\lambda_1|+]_{[2\Z_{\geq 0}]}$ are called the \textit{$(B_{-},0)$-DAG} and the \textit{$(\SL,0)$-DAG}, respectively.
If we want to emphasize that a slim bilateral DAG $\mathcal{Q}_{[2\Z_{\geq 0}]}$ is an $(X,0)$-DAG, we denote it as ${}_X\mathcal{Q}$ ($X\in\{B_+,B_-,\SL\}$).
Also, if we can use both ${}_{B_+}\mathcal{Q}$ and ${}_{B_-}\mathcal{Q}$ in some discussion, the symbol ${}_{B}\mathcal{Q}$ is used for convenience.
Under the restriction of 
${}_{\bullet}\mathcal{Q}_{[2\Z_{\geq 0}]}\simeq{}_{\bullet}[\lambda_1|\pm]_{[2\Z_{\geq 0}]}$ for $\bullet\in\Bits{1}$, the horizontal position (Cartan weight) of ${}_{X}\mathcal{Q}$ are defined by
\begin{align*}
{}_{B_{\bullet}}\mathcal{Q}^h
\rightarrow
\begin{cases}
{}_{\bullet}[\lambda_1|\pm)_{[h-h_0]}&h\geq h_0
\\
{}_{\bullet}(\pm|\bar\lambda_1]_{[h_0-h-2]}&h<h_0
\end{cases},
\ \ 
h_0:=
\begin{cases}
0&\operatorname{gap}({}_B\mathcal{Q})=1,\\
1&\operatorname{gap}({}_B\mathcal{Q})=-1
\end{cases},
\ \ 
{}_{\SL}\mathcal{Q}^h
\subseteq
{}_{B}\mathcal{Q}^h.
\end{align*}
Note that $\operatorname{Edge}({}_{\SL}\mathcal{Q})=\phi$ and ${}_{\SL}\mathcal{Q}^h\cong{}_{\SL}\mathcal{Q}^{-h}$.
More generally, the fat bilateral DAG $\mathcal{V}$ such that 
$\mathcal{V}
\simeq
{}_X\mathcal{Q}(\pm^{m}|\pm^{m})$
(resp.
$\mathcal{V}
\simeq
{}_X\mathcal{Q}[\pm^{m}|\pm^{m}]$)
for some $m\in\Z_{\geq 0}$ 
are called \textit{$(X,m)$-DAG}  (resp. \textit{fragmented $(X,m)$-DAG}).
The Cartan weights of these DAGs are the same as that of ${}_X\mathcal{Q}$.
We call $h_0$ above the \textit{minuscule weight}.
\end{definition}

Let us consider the bilateral extension of the case in Section \ref{subsection: derivation of Z-invariants}, i.e., the fragmented $(B,m)$-DAG
\begin{align*}
{}_{B}[\lambda_1|\pm][\pm^m|\pm^m]
\end{align*}
with minuscule weight $|\lambda_1|_{-}$.
As in \eqref{take restricted SL2 component from restricted B} and \eqref{take restricted B component from restricted SL2}, the fragmented $(\SL,m)$-DAG $[\lambda_1\pm^m|+\pm^m]_{[2\Z_{\geq 0}]}$ and its subgraph $[\lambda_1-\pm^{m-1}|+\pm^m]_{[2\Z_{\geq 0}]}=[\lambda_1-|+\pm][\pm^{m-1}|\pm^{m-1}]_{[2\Z_{\geq 0}]}$ can be considered, respectively.
In fact, under the restriction to $h=h_0$, they coincides with \eqref{take restricted SL2 component from restricted B} and \eqref{take restricted B component from restricted SL2}, so we use the same notations $\Gamma$, $\slash{\sim}$, and $\tilde{\Gamma}:=\slash{\sim}\circ\Gamma$ to the bilateral cases.
Then by \eqref{several formulas about gap},
\begin{align*}
\tilde{\Gamma}({}_{B}[\lambda_1|\pm][\pm^m|\pm^m])
=
{}_{B_+}[\lambda_1-|+\pm][\pm^{m-1}|\pm^{m-1}]
\end{align*}
is the fragmented $(B_+,m-1)$-DAG with minusucle weight $|\lambda_1|_{-}$ again.
In the same manner, for each $0\leq n\leq m$, $\Gamma$ sends the fragmented $(B,n)$-DAGs to the fragmented $(\SL,n)$-DAG, and $\slash{\sim}$ sends the fragmented $(\SL,n)$-DAG to the fragmented $(B_{+},n-1)$-DAGs, and the series of procedures keep the minuscule weight constant.
Finally we obtain the \textit{recursive expression} of the $(\SL,0)$-DAG 
\begin{align}
\label{recursive expression of starting socle nodes}
{}_{\SL}[\lambda_1+^m|+-^{m}]
\cong
{}_{\SL}[\lambda_1-^m|+^{m+1}]
=
\underbrace{\Gamma(\tilde{\Gamma}(\cdots\tilde{\Gamma}}_{\text{$(m+1)$-times}}({}_{B}[\lambda_1|\pm][\pm^m|\pm^m])\cdots)).
\end{align}
In particular, in the case $B=B_-$, the LHS of \eqref{recursive expression of starting socle nodes} is the \textit{terminal nodes} of ${}_B[\lambda_1|\pm][+^m|\pm^m]$.
\begin{remark}
\label{remark: from singlet to triplet}
The character of \eqref{recursive expression of starting socle nodes} is computed by taking the sum of (vertical shifted) \eqref{bosonic formula of the Z-invariants of Seifert 3-manifolds 2} with respect to the Cartan weights.
On the side of $\hat{Z}$-invariants/LCFTs, \eqref{character of general Fock module} and \eqref{character of general singlet module} corresponds to the ``\textit{Fock module}" and the irreducible module of the ``\textit{singlet Virasoro algebra}", respectively, whereas the horizontal extensions 
${}_{B}[\lambda_1|\pm][+^m|\pm^m]$ and ${}_{B}[\lambda_1|\pm][+^m|-^m]$, etc. correspond to the 
irreducible modules of the ``\textit{lattice VOA}" and the 
``\textit{triplet Virasoro algebra}", respectively.
\end{remark}

\section{An abelian categorification of $\hat{Z}$-invariants}
\label{section: Lie algebraic aspects}
\subsection{Annotated Loewy diagrams}
\label{section: annotated Loewy diagrams}
Let $M$ be an object of an abelian category.
We call $M$ is \textit{semisimple} if it is the direct sum of simple subobjects.
A \textit{Loewy series} of $M$ is a minimal length series among strictly ascending series of subobjects 
$0=M_0\subset M_1\subset\cdots\subset M_{l(M)}=M$
such that each successive quotient $M_k/M_{k-1}$ is semisimple. 
In this paper, the minimal length $l(M)$ is always finite.
We use the letters $\iota_k$ and $\pi_k$ for the embedding 
$\iota_k\colon M_k\hookrightarrow M$
and the projection 
$\pi_k\colon M\twoheadrightarrow M/M_k$, respectively.
Simple subobjects in each $M_k/M_{k-1}$ are called \textit{(simple) composition factors}, and by Jordan-H\"{o}lder theorem, they are unique up to isomorphisms.
For a fixed Loewy series of $M$, the \textit{Loewy diagram} is the diagram with $l(M)$ layers such that the $k$-th layer consists of the composition factors of $M_k/M_{k-1}$.
In addition, if the embedding of each composition factor $X_k$ into $M_k/M_{k-1}$ is also fixed, we call the Loewy series is \textit{annotated}.
For a fixed annotated Loewy series, the \textit{annotated Loewy diagram} is obtained by adding edges to the Loewy diagram according to the following rule: 
for each $1\leq k\leq l(M)$ and pair of composition factors $(X_k,X_{k-1})$, $X_k\subset M_k/M_{k-1}$ and $X_{k-1}\subset M_{k-1}/M_{k-2}$,
add the edge $X_k\rightarrow X_{k-1}$ if there exists a subquotient $X\subseteq M_k/M_{k-2}$ with a non-split short exact sequence $0\rightarrow X_{k-1}\rightarrow X\rightarrow X_k\rightarrow 0$.
In general, annotated Loewy diagrams are \textit{not} uniquely determined from mere Loewy series.
However, if all composition factors are distinct, then it is uniquely determined.
Let $C$ be the set of isomorphism classes of composition factors in $M$. By coloring each composition factor $X_k$ with its isomorphism class, we obtain the $C$-colored DAG from an annotated Loewy diagram.
Conversely, we can also consider the case where for a given $C$-colored DAG $\mathcal{Q}$, there exists an object $M$ with an annotated Loewy diagram represented by $\mathcal{Q}$ in the above sense.
When an annotated Loewy series of $M$ and a $C$-colored DAG $\mathcal{Q}$ are correpsonding in the above-mentioned sense, 
we write $M\sim\mathcal{Q}$, and
the structures of $M$ and $\mathcal{Q}$ corresponds as follows:
$M$ and $\mathcal{Q}$, 
a simple composition factor $X_k$ in the annotated Loewy series of $M$ and the corresponding node in $\mathcal{Q}$,
the subobject $\pi_k^{-1}(X_k)$ and the corresponding subgraph of $\mathcal{Q}$ (i.e., the induced subgraph generated by the node $X_k$),
the inclusion relation among $\pi_k^{-1}(X_k)$ and reachability in $\mathcal{Q}$,
and
the equivalence class of $X_k$ with respect to isomorphisms and the corresponding color $c(X_k)$ in $C$,
etc. 
As mentioned above, if $\mathcal{Q}$ is $C$-labeled, then the annotated Loewy diagram of $M$ is uniquely determined.
In this paper, we only consider the case $C=C_m$ for some $m\in\Z_{\geq 0}$.

\subsection{Defragmentation}
\label{subsection: defragmentation}
In the remainder of this paper, by using Freyd-Mitchell embedding theorem, we consider an object $M$ to be a left modules over a ring $R$ (unless otherwise noted, left and $R$ are omitted, and simply call $M$ a module below).
For a subset $S\subseteq M$, denote $\langle S\rangle\subseteq M$ the $R$-submodule generated by $S$.

\begin{definition}
Let $(M[\lambda])_{\lambda\in\Bits{n}}$ be a family of modules. 
A module $M$ is a \textit{deformation of $\bigoplus_{\lambda\in\Bits{n}}M[\lambda]$} if 
\begin{enumerate}
\item 
$M=\bigoplus_{\lambda\in\Bits{n}}M[\lambda]$ as vector spaces.
\item 
$\langle M[\lambda]\rangle\cap M[\mu]\not=0\Rightarrow\mu\leq\lambda$.
Denote $M[\leq\lambda]:=\langle M[\lambda]\rangle$.
\item 
$\langle M[\lambda]\rangle/(\langle M[\lambda]\rangle\cap\sum_{\mu<\lambda}\langle M[\mu]\rangle)\simeq M[\lambda]$ as modules.
Denote
$M[<\lambda]:=\langle M[\lambda]\rangle\cap\sum_{\mu<\lambda}\langle M[\mu]\rangle$.
\end{enumerate} 
For a submodule $N$ of $M$, denote the submodules 
$N[\leq\lambda]:=N\cap M[\leq\lambda]$, $N[<\lambda]:=N\cap M[<\lambda]$ of $N$.
Note that $N[\leq\lambda]/N[<\lambda]\simeq\langle N\cap M[\lambda]\rangle_{\text{old}}:=N[\lambda]$ as modules, where the RHS is the submodule of $M[\lambda]$ generated by $N\cap M[\lambda]$.
Therefore we sometimes use the same symbol $N[\lambda]$ for the subquotients $N[\leq\lambda]/N[<\lambda]$.
\end{definition}
In short, a deformation $M$ of $\bigoplus_{\lambda\in\Bits{n}}M[\lambda]$ is a (generally nontrivial) extension  of $(M[\lambda])_{\lambda\in\Bits{n}}$ along the partial order \eqref{partial ordering to define the quantization} such that the restriction to each $M[\lambda]$ is consistent with the original module structures.
Let us use such an ambiguity of module structures to define the ``reverse operation" of Definition \ref{definition: left/right delta cubizations}.

\begin{definition}
\label{def: defragmentation}
Let $\mathcal{Q}$ be a slim DAG and $(M[\lambda])_{\lambda\in\Bits{n}}$ be a family of modules such that $M[\lambda]\sim\mathcal{Q}[\lambda|\pm^n)(\pm^m|\pm^m)$ (resp. $M[\lambda]\sim\mathcal{Q}(\pm^n|\bar\lambda](\pm^m|\pm^m)$).
In particular, $\bigoplus_{\lambda\in\Bits{n}}M[\lambda]\simeq\mathcal{Q}[\pm^n|\pm^n)(\pm^m|\pm^m)$ (resp. $\bigoplus_{\lambda\in\Bits{n}}M[\lambda]\simeq\mathcal{Q}(\pm^n|\pm^n](\pm^m|\pm^m)$).
A deformation $M$ of $\bigoplus_{\lambda\in\Bits{n}}M[\lambda]$
is called \textit{defragmentation} if $M\sim\mathcal{Q}(\pm^{m+n}|\pm^{m+n})$.
Also, for a slim bilateral DAG $\mathcal{Q}_{[2\Z_{\geq 0}]}=\mathcal{Q}^{\mathrm{left}}_{[2\Z_{\geq 0}]}
\sqcup
\mathcal{Q}^{\mathrm{right}}_{[2\Z_{\geq 0}]}$
and a family of modules $(V[\lambda])_{\lambda\in\Bits{n}}$ 
such that
$V[\lambda]\sim\mathcal{Q}_{[2\Z_{\geq 0}]}[\lambda|\pm^n](\pm^m|\pm^m)$,
a deformation $V$ of $\bigoplus_{\lambda\in\Bits{n}}V[\lambda]$ is called \textit{defragmentation} if
$V\sim\mathcal{Q}_{[2\Z_{\geq 0}]}(\pm^{m+n}|\pm^{m+n})$
(i.e., for each $d\in\Z_{\geq 0}$, take the defragmentation of submodules $\bigoplus_{\lambda\in\Bits{n}}M[\lambda]_{[2d]}^{\mathrm{left}}\sim\mathcal{Q}^{\mathrm{left}}_{[2d]}[\pm^n|\pm^n)(\pm^m|\pm^m)$
(resp.
$\bigoplus_{\lambda\in\Bits{n}}M[\lambda]_{[2d]}^{\mathrm{right}}\sim\mathcal{Q}^{\mathrm{right}}_{[2d]}(\pm^n|\pm^n](\pm^m|\pm^m)$).
Note that by considering the reverse $\bar{\mathcal{Q}}$, we can defragment $(\pm^n|\pm^n]$ and $[\pm^n|\pm^n)$ on the left and right side, respectively.
We sometimes use the integral symbol for defragmentations and the corresponding annotated Loewy diagrams, e.g., 
\begin{align}
\label{integral symbol for defragmentation}
M
=\int_{\lambda\in\Bits{n}}M[\lambda]
\sim\int_{\lambda\in\Bits{n}}\mathcal{Q}_{[2\Z_{\geq 0}]}[\lambda|\pm^n](\pm^m|\pm^m)
=\mathcal{Q}_{[2\Z_{\geq 0}]}(\pm^{n+m}|\pm^{n+m}).
\end{align}
\end{definition}
\begin{remark}
\label{Fubini theorem and partial defragmentation}
Let $(M[\lambda])_{\lambda\in\Bits{N}}$ be a family of modles with annotated Loewy diagrams as in Definition \ref{def: defragmentation}.
In addition, let us assume that for each $I,I'\subseteq\{1,\dots,N\}$ such that $I\cap I'=\phi$, an iterated defragmentation $\int_{\lambda_{I}\in\Bits{I}}\int_{\lambda_{I'}\in\Bits{I'}}M[\lambda]$ is defined.
If ``Fubini's theorem" 
\begin{align}
\label{Fubini's theorem in our case}
\int_{\lambda_{I}\in\Bits{I}}\int_{\lambda_{I'}\in\Bits{I'}}M[\lambda]
=
\int_{\lambda_{I'}\in\Bits{I'}}\int_{\lambda_{I}\in\Bits{I}}M[\lambda]
=
\int_{\lambda_{I\sqcup I'}\in\Bits{I\sqcup I'}}M[\lambda],
\end{align} 
holds in any cases, 
then 
for each $I\subseteq\{1,\dots,N\}$ and $\lambda_{I}\in\Bits{I}$, the subquotient of $\int_{\lambda\in\Bits{N}}M[\lambda]$ corresponding to $\lambda_I\times\Bits{\bar{I}}$ is isomorphic to $\int_{\lambda_{\bar{I}}\in\Bits{\bar{I}}}M[\lambda_{I}\lambda_{\bar{I}}]$.
Namely, ``partial" defragmentations and gradings are compatible.
\end{remark}

Let us consider the defragmentations of the modules 
\begin{align}
\label{fragments appearing in LCFT contexts}
(\pm|\lambda_1][\lambda_{2\dots N}|\pm^{N-1})_{[2d]},
\   
[\bar{\lambda}_1|\pm)(\pm^{N-1}|\bar{\lambda}_{2\dots N}]_{[2d]}
\ \ \ 
(\lambda\in\Bits{N},\ d\in\Z_{\geq 0})
\end{align}
with the annotated Loewy diagrams given by the same symbols.
The modules
${}_{-}[\pm|\pm][\pm^{N-1}|\pm^{N-1}]_{[2\Z_{\geq 0}]}$
etc. shall also follow the same rule.
For each $\lambda_1\in\Bits{1}$, let us assume that a defragmentation 
\begin{align*}
V_{\lambda_1}
=
\int_{\lambda_{2\dots N}\in\Bits{N-1}}
{}_{-}[\lambda_1|\pm][\lambda_{2\dots N}|\pm^{N-1}]_{[2\Z_{\geq 0}]}
\sim
{}_{-}[\lambda_1|\pm](\pm^{N-1}|\pm^{N-1})_{[2\Z_{\geq 0}]}
\end{align*}
exists, where the RHS is the $(B_{-},N-1)$-DAG. 
By considering extensions of these two modules $V_{\lambda_1}$ ($\lambda_1\in\Bits{1}$), let us try to make modules $P_+,P_-$ with $(\SL,N)$-DAG of the minuscule weights $h_0=0,1$, respectively, as annotated Loewy diagrams.
The first and obvious answer is to consider the defragmentation with respect to $\lambda_1$.
Namely, $P_{-}=\int_{\lambda_1\in\Bits{1}}V_{\lambda_1}$ has the $(\SL,N)$-DAG $[\ |\ ]_{[2\Z_{\geq 0}]}(\pm^N|\pm^N)$ of $h_0=1$ as an annotated Loewy diagram by an extension $0\rightarrow V_{-}^{h}\rightarrow P_{-}^h\rightarrow V_{+}^{h-1}\rightarrow 0$ ($h\in 2\Z+1$).
On the other hand, if one tries to make a module $P_+$ with the $(\SL,N)$-DAG of $h_0=0$ as an annotated Loewy diagram by an extension $0\rightarrow V_{+}^{h}\rightarrow P_{+}^h\rightarrow V_{-}^{h-1}\rightarrow 0$ ($h\in 2\Z$), the hypercube corresponding to the trivial $\SL$-module $\C$ is ``lacking" because there are not enough simple composition factors.
Namely, we consider the module $P_{+}\simeq\bigoplus_{d\in\Z_{\geq 0}}\C^{2d+1}\otimes W_{2d}$ such that
$W_{2d}\sim(\pm|\pm)_{[-]}(\pm^{N-1}|\pm^{N-1})_{[2(d-1)]}$ for $d\geq 1$, but
$W_0\sim((\pm|\pm)_{[-]}\setminus(+|+)_{[-]})(\pm^{N-1}|\pm^{N-1})_{[-2]}$.
Under the same notation,  $P_-\simeq\bigoplus_{d\in\Z_{\geq 1}}\C^{2d}\otimes W_{2d-1}$ and $W_{2d-1}\sim(\pm^N|\pm^N)_{[2(d-1)]}$.
Note that such a ``lacking defragmentation" for ${}_\bullet[\pm^n|\pm^n](\pm^m|\pm^m)$ ($\bullet\in\Bits{1}$) happens only for the case as above, i.e., $n=1$ and $\C\otimes W_{0}$.
\begin{remark}
\label{remark: why one direction is reversed?}
One of the reasons for considering only the first parameter $\lambda_1$ ``reversed" in \eqref{fragments appearing in LCFT contexts} is to be consistent with the data on the LCFT side (see also the second half of Section \ref{subsection: derivation of Z-invariants} and Remark \ref{remark: from singlet to triplet}). 
In fact, for $N=1,2$, the DAGs in \eqref{fragments appearing in LCFT contexts} is consistent with the Virasoro module structure of the corresponding Fock modules (i.e., the case where the ring $R$ is the universal enveloping algebra of the Virasoro algebra, see e.g., \cite[Section 2]{FGST}).
Under this correspondence, the modules $P_{\pm}$ are the projective covers of the corresponding simple modules of the ``triplet Virasoro algebra" in the sense of Remark \ref{remark: from singlet to triplet} (cf. \cite[Figure 6.3]{N}).
This provides a more plausible explanation for the above-mentioned reversal of $\lambda_1$. 
In fact, if no such reversal exists and $[\lambda|\pm^N)_{[2d]}$,  $(\pm^N|\bar\lambda]_{[2d]}$ are treated as ``Fock modules”, there are $N$ possible choices of $P_+$ depending on which parameter $\lambda_i$ is glued at the last step in the previous paragraph.
However, it contradicts to the uniqueness of the projective cover of a simple module.
\end{remark}
Let us combine the recursive computation in the last paragraph of Section \ref{subsection: Cartan extension} with the discussion above.
If there are defragmentations at each step of the computation, then each of the operations $\Gamma$ and $/\sim$ can be regarded as those of taking the $(\SL,m)$-DAG from $(B,m)$-DAG and $(B,m)$-DAG from $(\SL,m+1)$-DAG, respectively.
Put another way, starting from the ``projective cover" $P_{\pm}$ corresponding to the $(\SL,N)$-DAGs, the above two operations $\Gamma,/\sim$ are repeated to iterate through the $\SL$-type and $B$-type, gradually decreasing their dimensions, until finally we obtain the ``simple module" corresponding to the $(\SL,0)$-DAGs.
In the next section, we will see that if a defragmentation $[\lambda_1-^n|+^n\pm](\pm^m|\pm^m)_{[2\Z_{\geq 0}]}$ actually has  a $B$-action compatible with the $R$-module structure, it is an example of the \textit{shift system} \cite{LS}, and the operation $\Gamma$ is that of taking the induced $\SL$-module (or taking the global section functor $H^0(\SL\times_B-)$, i.e., \textit{Feigin-Tipunin construction}).
In particular, if such a situation holds at each step, it gives a \textit{recursive shift system} and \textit{nested FT construction}.

\subsection{Recursive shift systems}
\label{subsection: recursive shift systems}
In Definition \ref{definition of the shift system}, Theorem \ref{main theorem for shift system} below, we use the notation in \cite[Section 1.1]{LS}.
\begin{definition}\label{definition of the shift system}
\cite[Definition 1.2]{LS}
The \textit{shift system} refers to a triple
$(\Lambda,\uparrow,\{V_\lambda\}_{\lambda\in\Lambda})$    
that satisfies the following:
\begin{enumerate}
\item\label{FT data 1}
$\Lambda$ is a $W$-module.
Denote $\ast\colon W\times\Lambda\rightarrow\Lambda$ the $W$-action on $\Lambda$.
\item\label{FT data 2}
$\uparrow\colon W\times\Lambda\rightarrow P$ is a map (\textit{shift map}) such that
\begin{enumerate}
\item\label{associativity of carry-over}
$\sigma_i\sigma\uparrow\lambda=\sigma_i(\sigma\uparrow\lambda)+\sigma_i\uparrow(\sigma\ast\lambda)$.
\item\label{length and carry-over}
If $l(\sigma_i\sigma)=l(\sigma)+1$, then $(\sigma\uparrow\lambda,\alpha_i^\vee)\geq 0$.
\item\label{-1 property of carry-over}
If $\lambda\not\in\Lambda^{\sigma_i}$, then $(\sigma_i\uparrow\lambda,\alpha_i^\vee)=-1$.
If $\lambda\in\Lambda^{\sigma_i}$, then $\sigma_i\uparrow\lambda=-\alpha_i$.
\end{enumerate}
\item\label{FT data 3}
$V_\lambda$ is a weight $B_{-}$-module such that 
\begin{enumerate}
\item\label{conformal grading of FT data}
$V_\lambda=\bigoplus_{\Delta}V_{\lambda,\Delta}$ and each $V_{\lambda,\Delta}$ is a finite-dimensional weight $B_{-}$-submodule.
\item\label{Felder complex of FT data}
For any $i\in\Pi$ and $\lambda\not\in\Lambda^{\sigma_i}$, there exists $P_i$-submodules $W_{i,\lambda}\subseteq V_{\lambda}$, $W_{i,\sigma_i\ast\lambda}\subseteq V_{\sigma_i\ast\lambda}$ and a $B_{-}$-module homomorphism 
$Q_{i,\lambda}\colon V_\lambda\rightarrow W_{i,\sigma_i\ast\lambda}(
\sigma_i\uparrow\lambda)$ such that we have the short exact sequence 
\begin{align*}
0\rightarrow W_{i,\lambda}
\rightarrow V_{\lambda}
\xrightarrow{Q_{i,\lambda}} W_{i,\sigma_i\ast\lambda}(\sigma_i\uparrow\lambda)
\rightarrow 0
\end{align*}
of $B_{-}$-modules, called \textit{Felder complex}.
If $\lambda\in\Lambda^{\sigma_i}$, then $V_\lambda$ has the $P_i$-module structure.
\end{enumerate}
\end{enumerate}
For a weight $B_{-}$-module $V$ in a shift system, $H^0(G\times_BV)$ is called the \textit{Feigin-Tipunin construction}.
\end{definition}
\begin{theorem}
\label{main theorem for shift system}
\cite[Theorem 1.1, Corollary 1.15]{LS}
Let $(\Lambda,\uparrow,\{V_\lambda\}_{\lambda\in\Lambda})$ is a shift system.
\begin{enumerate}
\item\label{main theorem 1(1)}
\textit{(Feigin-Tipunin conjecture/construction)}
The evaluation map 
\begin{align*}
\operatorname{ev}\colon H^0(G\times_BV_\lambda)\rightarrow\bigcap_{i\in\Pi}\ker Q_{i,\lambda}|_{V_\lambda},\ \ s\mapsto s(\operatorname{id}_{G/B})
\end{align*}
is injective, and is isomorphic iff $\lambda\in\Lambda$ satisfies the following condition:
\begin{align}
\text{For any $(i,j)\in\Pi\times\Pi$, $\lambda\in\Lambda^{\sigma_j}$ or $(\sigma_j\uparrow\lambda,\alpha_i^\vee)=-\delta_{ij}$.}
\end{align}
Note that the image of $H^0(G\times_BV_\lambda)$ is the maximal $G$-submodule (i.e, the maximal $B$-submodule such that its $B$-action can be extended to the $G$-action)
of $V_\lambda$.
\item\label{main theorem 1(2)}
\textit{(Borel-Weil-Bott type theorem)}
For a minimal expression $w_0=\sigma_{i_{l(w_0)}}\cdots\sigma_{i_1}\sigma_{i_0}$ of the longest element $w_0$ of $W$ (where $\sigma_{i_0}=\operatorname{id}$ for convenience), if $\lambda\in\Lambda$ satisfies the following condition:
\begin{align}
\text{
For any $0\leq m\leq l(w_0)-1$,
$(\sigma_{i_m}\cdots\sigma_{i_0}\uparrow\lambda,\alpha_{i_{m+1}}^\vee)=0$,
}
\end{align}
then we have a natural $G$-module isomorphism
\begin{align*}
H^n(G\times_BV_\lambda)\simeq H^{n+l(w_0)}(G\times_BV_{w_0\ast\lambda}(w_0\uparrow\lambda)).
\end{align*}
In particular, $H^{n>0}(G\times_BV_\lambda)\simeq 0$ and we have the  character formula
\begin{align*}
\ch_q H^0(G\times_BV_\lambda)
=
\sum_{\beta\in P_+}\operatorname{dim}L_{\beta}
\sum_{\sigma\in W}(-1)^{l(\sigma)}\ch_q V_{\sigma\ast\lambda}^{h=\beta-\sigma\uparrow\lambda}
\end{align*}
for $\ch_qV:=\sum_{\Delta}\dim V_\Delta q^\Delta$ and the irreducible $\g$-module $L_\beta$ with highest weight $\beta$.
\end{enumerate}
\end{theorem}
In this paper, we only consider the case $\g=\sll_2$, $\Lambda=\{\lambda,\sigma_1\ast\lambda\}\simeq\Bits{1}$.
Then the Felder complex is
\begin{align}
\label{B-minus Felder complex}
0
\rightarrow
W_{\pm}
\rightarrow
V_{\pm}
\rightarrow
W_{\mp}(-1)
\rightarrow
0,
\end{align}
where $W_{\pm}$ is the maximal $\SL$-submodules of the weight $B_{-}$-modules $V_{\pm}$, respectively.
By Theorem \ref{main theorem for shift system}, 
\begin{align*}
H^0(\SL\times_{B_{-}}V_{\pm})
\simeq 
W_{\pm},
\ \ 
H^1(\SL\times_{B_{-}}V_{\pm})
\simeq 
0
\end{align*}
and the character formula is given by
\begin{align}
\label{B-minus character formula}
\ch_{q}H^0(\SL\times_{B_{-}}V_{\pm})^{h=m}
=
\sum_{n\geq m}
\ch_qV_{\pm}^{h=n}
-
\ch_qV_{\mp}^{h=n+1}.
\end{align}
On the other hand, we can consider the $B_{+}$-version of Definition \ref{definition of the shift system} and Theorem \ref{main theorem for shift system}. 
The Felder complex 
\begin{align}
\label{B-plus Felder complex}
0
\rightarrow
W_{\pm}
\rightarrow
V_{\pm}
\rightarrow
W_{\mp}(1)
\rightarrow
0
\end{align}
is a short exact sequence of $B_+$-modules,
where $W_{\pm}(1)$ is the maximal $\SL$-submodule of $V_{\pm}(1)$, respectively.
In the same manner as $B_{-}$-case, we have
\begin{align*}
H^0(\SL\times_{B_{+}}V_{\pm})
\simeq 
V_{\pm}/W_{\pm}
\simeq
W_{\mp}(1),
\ \ 
H^1(\SL\times_{B_{+}}V_{\pm})
\simeq 
0
\end{align*}
and the character formula of $H^0(\SL\times_{B_+}V_{\pm})$ is the same as
\eqref{B-minus character formula}.
Denote the $\SL$-module structures by
\begin{align*}
W_{\pm}
\simeq
\bigoplus_{n\geq 0}L_n\otimes W_{\pm,n}
\simeq
\bigoplus_{n\geq 0}\C^{n+1}\otimes W_{\pm,n},
\end{align*}
i.e., $W_{\pm,n}$ is the multiplicity space of  a weight vector of the $(n+1)$-dimensional simple $\SL$-module $L_n\simeq\C^{n+1}$.
By abuse of notation, we sometimes use the same symbol $W_{\mp,n}$ for the corresponding multiplicity space in $V_{\pm}$.

Let us combine the above discussion with Section \ref{section: annotated Loewy diagrams}.
Let $V_{\pm}$ be $R$-modules such that $V_{\pm}$ consists a shift system with respect to a $B$-action commuting with the $R$-action.
Then $W_{\pm}$ and $V_{\pm}/W_{\pm}\simeq W_{\mp}$ are direct sums of $R$-modules $(W_{\pm,n})_{n\geq 0}$ and $(W_{\mp,n})_{n\geq 0}$, respectively.
Let us consider an annotated Loewy diagram of $V_{\pm}$ such that each $W_{\pm,n}$ and $W_{\mp,n}$ are subgraphs.
Then by the compatibility of the $B$-action, in the annotated Loewy diagrams of each $V_{\pm}^h$, there are no edges other than $W_{\pm,n}\rightarrow W_{\mp,n\pm 1}$ (i.e., for a pair of simple composition factors $(X_k,X_{k+1})$, $X_k\rightarrow X_{k+1}$ only if $X_k$ and $X_{k+1}$ are in annotated Loewy diagrams of $W_{\pm,n}$ and $W_{\mp,n\pm 1}$ in that of $V_{\pm}^h$ for some $n$ and $h$, respectively). 

The $(B_{\bullet},m)$-DAGs (resp. $(\SL,m)$-DAGs) in Definition \ref{definition: X-type DAGs} are examples of annotated Loewy diagrams of weight $B_\bullet$-modules $V_{\pm}$ (resp. the $\SL$-modules $W_\pm$) belonging to a ($B_\bullet$-type) shift system.
On the other hand, by considering the compatibility of the $B$-action, fragmented $(X,m)$-DAGs and each $\lambda$-fragments are not annotated Loewy diagrams of a shift system except for $m=0,1$.
As mentioned in Section \ref{section: technical summary} and the last paragraph of Section \ref{subsection: defragmentation}, our aim is to use the ambiguity of module structures to defragment the fragmented $(X,m)$-DAGs so that it might be a shift system.
Unfortunately, the author is still unsure at this time how to construct $X$-actions that would form a recursive shift system from given data such as $\lambda$-fragments or their defragmentations etc. (see also Remark \ref{remark: how to construct the B-actions?} below).
However, on the contrary, if a deformation of the fragmented $(X,m)$-modules form a shift system, then such deformations are to some extent determined.

\begin{lemma}
\label{lemma: shift system roughly determines the structure}
Let $V_{\lambda_1}$ be a deformation of $\bigoplus_{\lambda_{2\dots N}\in\Bits{N-1}}{}_{-}[\lambda_1|\pm][\lambda_{2\dots N}|\pm^{N-1}]$ such that 
$V_{\lambda_1}$ consists a shift system with respect to $\lambda_1\in\Bits{1}$.
Then for each $\lambda_1\in\Bits{1}$,
\begin{align*}
[W_{\lambda_1,h}]=
\begin{cases}
[(\lambda_1|+)(\pm^{N-1}|\pm^{N-1})_{[h-|\lambda_1|_{-}]}]
&h\in 2\Z_{\geq 0}+|\lambda_1|_{-}
\\
0
&h\in 2\Z_{\geq 0}+|\bar{\lambda}_1|_{-}
\end{cases}
\end{align*}
in the Grothendieck group, and
$V_{\lambda_1}$ has an annotated Loewy diagram such that for any $\lambda\in\Bits{N-1}$,  subquotients isomorphic to ${}_{-}[\lambda_1|\pm](\lambda|\pm^{N-1})$ and ${}_{-}[\lambda_1|\pm](\pm^{N-1}|\bar\lambda)$ are contained.
\footnote{Note that if $N\geq 3$, this result does not necessarily lead to $W_{\lambda_1,h}\sim
(\lambda_1|+)(\pm^{N-1}|\pm^{N-1})_{[h-|\lambda_1|_{-}]}$.}
\end{lemma}
\begin{proof}
The basic idea is to duplicate the structures of $h\geq 0$ and $h<0$ each other by comparing them using the $\SL$-action.
Let us fix $\lambda_1$ and $h\geq 0$, and focus on the segments $\bigoplus_{\lambda\in\Bits{N-1}}(+|\lambda_1)(\lambda|\pm^{N-1})_{[h-|\lambda_1|_{-}]}$ 
and 
$\bigoplus_{\lambda\in\Bits{N-1}}(\bar{\lambda}_1|-)(\pm^{N-1}|\bar\lambda)_{[h-|\bar{\lambda}_1|_{-}-1]}$ in $V_{\lambda_1}^h$ and $V_{\lambda_1}^{-h}$, respectively.
In the case $h\geq 1$, the $\SL$-action sends the leftmost segment 
$(\bar{\lambda}_1|-)(\pm^{N-1}|+^{N-1})_{[h-|\bar{\lambda}_1|_{-}-1]}$
in the multiplicity space $W_{\lambda_1,h}\subseteq V_{\lambda_1}^{-h}$ of the lowest weight vector  
to the segment 
$\bigoplus_{\lambda\in\Bits{N-1}}(+|\lambda_1)(\lambda|+^{N-1})_{[h-|\lambda_1|_{-}]}$
spanning the top composition factors of the fragments ${}_{-}(\pm|\lambda_1][\lambda|\pm^{N-1})_{[h-|\lambda_1|_{-}]}$ in the multiplicity space
$W_{\lambda_1,h}\subseteq V_{\lambda_1}^{h}$ of the highest weight vector.
In particular, the later also contains (the module isomorphic to) $\bigoplus_{\lambda\in\Bits{N-1}}(+|\lambda_1)(\lambda|\pm^{N-1})_{[h-|\lambda_1|_{-}]}$ as well, and thus $[W_{\lambda_1,h}]\geq[(\lambda_1|+)(\pm^{N-1}|\pm^{N-1})_{[h-|\lambda_1|_{-}]}]$ in the Grothendieck group.
On the other hand, by applying the discussion above to all $h\geq 1$, we have $[W_{\lambda_1,h}]=[(\lambda_1|+)(\pm^{N-1}|\pm^{N-1})_{[h-|\lambda_1|_{-}]}]$ from the semisimplicity of $\SL$-modules.
In the case $h=0$ (and thus $\lambda_1=+$),  replacing the above duplication by the $\SL$-action by that of the Felder complex shows the  similar results for $W_{+,0}$.
Thus, the first half of the assertion is proved.

For a fixed $\lambda_1$, the multiplicity spaces $W_{\bar{\lambda}_1,h+1}$ in the quotient $V_{\lambda_1}/W_{\lambda_1}$ satisfy 
\begin{align*}
[W_{\bar{\lambda}_1,h+1}]
=[(\bar{\lambda}_1|+)(\pm^{N-1}|\pm^{N-1})_{[h+1-|\bar{\lambda}_1|_{-}]}]
=[(-|\lambda_1)(\pm^{N-1}|\pm^{N-1})_{[h-|\lambda_1|_{-}]}]   
\end{align*}
and contain the modules isomorphic to $\bigoplus_{\lambda\in\Bits{N-1}}(-|\lambda_1)(\lambda|\pm^{N-1})_{[h-|\lambda_1|_{-}]}$ by the above discussion.
By the uniqueness of these left segments in $W_{\lambda_1,h}$ and $W_{\bar{\lambda}_1,h+1}$, 
the arrows from (the composition factors isomorphic to) $(-|\lambda_1)(\lambda|\pm^{N-1})_{[h-|\lambda_1|_{-}]}$ to 
$(+|\lambda_1)(\lambda|\pm^{N-1})_{[h-|\lambda_1|_{-}]}$ $(\lambda\in\Bits{N-1})$ consist the $N$-cubes $(\pm|\lambda_1)(\lambda|\pm^{N-1})_{[h-|\lambda_1|_{-}]}$.
In other words, we obtain the arrows from  $W_{\bar{\lambda}_1,h+1}$ to $W_{\lambda_1,h}$  in $V_{\lambda_1}$.
By applying the discussion so far in this proof to $-h\leq -1$, we can also obtain the arrows from  $W_{\bar{\lambda}_1,h-1}$ to $W_{\lambda_1,h}$ in $V_{\lambda_1}$.
By combining these arrows, we obtain the subquotients isomorphic to ${}_{-}[\lambda_1|\pm](\lambda|\pm^{N-1})$ for each $\lambda\in\Bits{N-1}$.
The remaining part is proved in the same manner, and thus the second half of the assertion is also proved.
\end{proof}
By Lemma \ref{lemma: shift system roughly determines the structure}, the structure of shift system roughly determines that of  $V_{\lambda_1}$, in particular, the operation $\Gamma$ in the last paragraph of Section \ref{subsection: Cartan extension} corresponds to the global section functor $H^0(\SL\times_B-)$.
Therefore, by combining the above discussions, we achieve at the recursive shift systems and nested FT constructions.

\begin{definition}
\label{definition: recursive shift systems and nested FT constructions in our cases}
The modules $(V_{\lambda_1})_{\lambda_1\in\Bits{1}}$ 
consists a \textit{recursive shift system}
if 
for any $0\leq n\leq N-1$ and element of the symmetric group $s\in\mathfrak{S}_{N-1}$ permuting $\{2,\dots,N\}$, 
the subquotient
\begin{align}
\label{all shift system in the ecursive shift systems}
\underbrace{
\tilde{H}^0(\SL\times_{B_+}
\tilde{H}^0(\SL\times_{B_+}
\dots
\tilde{H}^0(\SL\times_{B_-}}_{\text{$n$-times}}
V_{\lambda_{1}})_{{s(2)}})\dots)_{{s(n+1)}}
\end{align}
corresponding to 
${s}_{-}[\lambda_1|+][-^{n}\pm^{N-n-1}|+^{n-1}\pm\pm^{N-n-1}]$
is a shift shistem with respect to 
$\lambda_{s(n)}\in\Bits{1}$,
where ${s}_{-}[\lambda_1|\mu_1][\lambda_{2\dots N}|\mu_{2\dots N}]
:=
{}_{-}[\lambda_{1}|\mu_{1}][\lambda_{s(2)\dots s(N-1)}|\mu_{s(2)\dots s(N)}]$.
The $\SL$-modules obtained by applying the global section functor $H^0(\SL\times_B-)$ to \eqref{all shift system in the ecursive shift systems} are collectively referred to as \textit{nested Feigin-Tipunin constructions}.
\end{definition}
If $(V_{\lambda_1})_{\lambda_1\in\Bits{1}}$ 
consists a recursive shift system,
then \eqref{recursive expression of starting socle nodes} is represented by the nested FT constructions
\begin{align}
\label{nested Feigin-Tipunin construction in the main text}
{}_{-}[\lambda_1|+][+^m|-^{m}]
\simeq
\underbrace{
{H}^0(\SL\times_{B_+}
\tilde{H}^0(\SL\times_{B_+}
\dots
\tilde{H}^0(\SL\times_{B_-}}_{\text{$N$-times}}
V_{\lambda_{1}})_{{s(2)}})\dots)_{{s(N)}})
\end{align}
and the character is computed by recursively applying \eqref{B-minus character formula} following the procedure in Section \ref{subsection: derivation of Z-invariants}.

\begin{corollary}
\label{recursive shift system is defragmentation}
If $(V_{\lambda_1})_{\lambda_1\in\Bits{1}}$ 
consists a recursive shift system, then all \eqref{all shift system in the ecursive shift systems} are defragmentations.
\end{corollary}
\begin{proof}
By the induction of $N$, suppose that the assertion is already proved for $m\leq N-1$. 
By Lemma \ref{lemma: shift system roughly determines the structure}, it suffices to show that the multiplicity spaces $W_{\lambda_1,h}$ of the weight vectors in $V_{\lambda_{1}}$ satisfy
\begin{align}
\label{each weight vector is hypercubic}
W_{\lambda_1,h}
\sim
(\lambda_1|+)(\pm^{N-1}|\pm^{N-1})_{[h-|\lambda_1|_{-}]}.
\end{align}
By applying the assumption of the induction to the subquotient $\tilde{H}^0(\SL\times_{B_{-}}V_{\lambda_1})_{s(2)}$ in \eqref{all shift system in the ecursive shift systems} for each $\lambda_1\in\Bits{1}$ and $s\in\mathfrak{S}_{N-1}$, we see that the annotated Loewy diagram of $W_{\lambda_1,h}$ must contain $2(N-1)$-cubes $s(\lambda_1|+)(-\pm^{N-2}|\mu_2\pm^{N-2})_{[h-|\lambda_1|_{-}]}$ ($\mu_2\in\Bits{1}$).
By combining this fact with Lemma \ref{lemma: shift system roughly determines the structure} for $W_{\lambda_1,h}$, \eqref{each weight vector is hypercubic}
is proved.
\end{proof}

\begin{remark}
\label{remark: how to construct the B-actions?}
The question of how to construct a $B$-action such that a given deformation or defragmentation defines a shift system seems non-trivial.
On the LCFT side, in the case $N=1$, such a $B$-action is easily obtained by considering the zero-modes preserving the conformal weights, called long screening operators (for the most general discussion in the principal nilpotent cases, see \cite[Section 2.4]{LS}).
However, already in the case of $N=2$, such a $B$-action is given by a complicated integral \cite[Section 4.5]{TW} related to \textit{Lusztig's divided power} \cite[Section 3]{FGST}\footnote{Even in the case of $N=1$, the Lusztig's divided power of the short screening operators is required the to consider the $B_+$-action.}.
Furthermore, as mentioned above, the Virasoro module structures of the Fock modules give the $\lambda$-fragments in the cases of $N=1,2$ \cite[Section 2]{FGST}, but no such decompositions are known in the case of $N\geq 3$.
It is a future work to come up with a mechanism to generate such long/short screening operators.
\end{remark}

In conclusion, let us provide an expected abelian categorification of $\hat{Z}$-invariants for Seifert $3$-manifolds with $(N+2)$-exceptional fibers. 
 In other words, it is an abelian category with the shift systems equipped with the compatible ``fragmentations" and ``defragmentations" of them along the structure of the hypercube graph.
\begin{definition}
\label{definition: hypercubic categorification}
For $N\in\Z_{\geq 1}$, an abelian category (of $R$-modules) $\mathcal{C}_N$ is called \textit{hypercubic categorification} (of $\hat{Z}$-invariants for Seifert $3$-manifolds with $(N+2)$-exceptional fibers) if
\begin{itemize}
\item
${}_{-}[\pm|\pm][\pm^{N-1}|\pm^{N-1}]=\bigoplus_{\lambda_1\cdots\lambda_N\in\Bits{N}}{}_{-}[\lambda_1|\pm][\lambda_{2\dots N}|\pm^{N-1}]\in\mathcal{C}_N$ (see Section \ref{subsection: Cartan extension}).
\item 
For any $I,I'\subseteq\{2,\dots,N\}$ such that $I\cap I'=\phi$, the defragmentation
$\int_{\lambda_I\in\Bits{I}}\int_{\lambda_{I'}\in\Bits{I'}}{}_{-}[\lambda_1|\pm][\lambda_{2\dots N}|\pm^{N-1}]$ satisfying the Fubini's theorem exists (see Remark \ref{Fubini theorem and partial defragmentation}).
\item 
The modules $(V_{\lambda_1}:=\int_{\lambda_{2\dots N}\in\Bits{N-1}}{}_{-}[\lambda_1|\pm][\lambda_{2\dots N}|\pm^{N-1}])_{\lambda_1\in\Bits{1}}$ consists a recursive shift system (see Definition \ref{definition: recursive shift systems and nested FT constructions in our cases}), and the ($\SL$-)modules $P_{\pm}$ in Section \ref{subsection: defragmentation} also exists.
\end{itemize}
\end{definition}
As detailed in the next subsection, such an abelian category can be viewed as the result of applying the forgetful functor to the representation category of an expected LCFT corresponding to the Seifert $3$-manifold.
Although such LCFTs have not yet been constructed in general cases, both the shift system theory and the hypercubic structure of the category that allows its recursive application are abelian categorical, providing a hypothetical and prototypical, but unified construction/research method for such LCFTs (see Section \ref{subsection: background, motivation and program}).

\subsection{Glimpse of LCFTs for Seifert $3$-manifolds}
\label{subsection: glimpse}
Let us give the correspondence between our discussion and hypothetical LCFTs for Seifert $3$-manifolds with $(N+2)$-exceptional fibers, which we have described piecemeal so far (recall the last paragraph of Section \ref{subsection: derivation of Z-invariants}, and Remark \ref{remark: from singlet to triplet}, \ref{remark: why one direction is reversed?}, \ref{remark: how to construct the B-actions?}, etc.).
For $N=1,2$, the corresponding LCFTs are already known respectively (e.g., \cite{LS,CNS} for $N=1$, \cite{FGST,TW,N} for $N=2$), from which we can infer what the representation theories of LCFTs corresponding to general $N$ are.

In this paper, we consider the case $\g=\sll_2$.
Let $p_1,\dots,p_N\in\Z_{\geq 2}$ are coprime integers.
This integer sequence $(p_1,\dots,p_N)$ represents the ``level" of the LCFT. 
For example, for $N=1,2$, the levels of the LCFTs are $1/p_1$ and $p_2/p_1$, respectively
\footnote{There are situations where $p_1$ or ${p_1}/{p_2}$ might be considered by Feigin-Frenkel duality, but the details are not considered here.}.
Of course, it is unclear what ``number" can be interpreted as the level in the case of $N\geq 3$, but we shall for the moment consider this $(p_1,\dots,p_N)$ as the level.
In correspondence with the LCFTs, the ring $R$ is intended to be the ``universal enveloping algebra of the Virasoro VOA (in general, W-algebra) at level $(p_1,\dots,p_N)$", and the nodes of each DAG (or annotated Loewy diagram) represent simple composition factors of such ``Virasoro VOA".
For $N=1,2$, the Virasoro VOAs are embedded in the Heisenberg VOAs (Fock modules) with the same levels, and simple modules of the (rescaled root) lattice VOAs are obtained by aligning the Fock modules along the lattice $\sqrt{p_1\dots p_N}Q_{\sll_2}=\sqrt{2p_1\cdots p_N}\Z$.
Recall that the simple modules of such a lattice VOA is parametrized by the finite set $(\sqrt{2p_1\cdots p_N}\Z)^\ast/\sqrt{2p_1\cdots p_N}\Z$ with $2p_1\cdots p_N$ representatives
\begin{align*}
(\lambda_1,r_1,\cdots,r_N)
\ \ \ \ 
(\lambda_1\in\Bits{1},
\ 
1\leq r_i\leq p_i).
\end{align*}
Let us fix $(r_1,\dots,r_N)$ such that $1\leq r_i<p_i$ for each $1\leq i\leq N$.
In this case, the nodes colored by $(\lambda,d)\in\Bits{N}\times\Z_{\geq 0}$ in the DAGs or annotated Loewy diagrams shall represent simple composition factors parametrized by the representative
$(\lambda_1,\lambda_1r_1,\cdots,\lambda_Nr_N)$, where $-r_i$ intends to $p_i-r_i$,  and with the top ``conformal weight" $\Delta_{\lambda,d}$ in  \eqref{character of general Fock module}. 
From known results for $N=1,2$, the following  ``dictionary" can be inferred.
\renewcommand{\arraystretch}{1.2}
\begin{table}[htbp]
\centering
\begin{tabular}{|l|l|} \hline
\textbf{DAGs (annotated Loewy diagrams)} & \textbf{logarithmic CFTs}
\\ \hline\hline
${}_{-}[\lambda_1|\pm)[\lambda_{2\cdots N}|\pm^{N-1}]_{[2n]}$, 
${}_{-}(\pm|\bar{\lambda}_1][\lambda_{2\cdots N}|\pm^{N-1}]_{[2n]}$
& ``Fock modules"
\\ \hline
${}_{-}[\lambda_1|\pm][\lambda_{2\cdots N}|\pm^{N-1}]$ 
& ``simple modules of the lattice VOAs"
\\ \hline
${}_{-}(+|+)(+^{N-1}|\pm^{N-1})$, $\forall r_i=1$
&``universal Virasoro VOA" (i.e., $R$)
\\ \hline
${}_{-}(+|+)(+^{N-1}|\pm^{N-1})
\cup
{}_{-}[+|+)(+^{N-1}|-^{N-1})$, $\forall r_i=1$
&``singlet Virasoro VOA" 
\\ \hline
${}_{-}(+|+)(+^{N-1}|\pm^{N-1})
\cup
{}_{-}[+|+](+^{N-1}|-^{N-1})$, $\forall r_i=1$
&``triplet Virasoro VOA" 
\\ \hline
${}_{-}[\lambda_1|+](+^{N-1}|-^{N-1})$
&``simple modules of the triplet Virasoro VOA" 
\\ \hline
${}_{-}(\lambda_1|+)(+^{N-1}|\mu)_{[k]}$, $0\leq k< N-1-|\mu|_{-}-|\lambda_1|_{-}$
&``simple modules of the (triplet) Virasoro VOA" 
\\ \hline
${}_{-}[\lambda_1|+](+^{N-1}|-^{N-1})
\subset 
P_{\lambda_1}$ (see Section \ref{subsection: defragmentation})
&``projective covers of the simple modules" 
\\ \hline
$\int{}_{-}[\lambda_1|\mu_1][\lambda_{2\cdots N}|\mu_{2\cdots N}]$, $\forall\mu_i\in\{+,-,\pm\}$
&``logarithmic modules of the triplet Virasoro VOA" 
\\ \hline
\end{tabular}
\caption{Glimpse of the ``dictionary" between DAGs and LCFTs ($\forall r_i<p_i$).}
\label{table:dictionary}
\end{table}

As mentioned in Section \ref{section: technical summary}, if we consider the bits that parametrize the simple composition factors as variables, then the reduction of variables by defragmentation $\int$ in Definition \ref{def: defragmentation} is done by considering the corresponding $2n$-cube graphs (resp. $(2n+1)$-cube graphs) in DAGs as new ``nodes" (resp. ``edges").
In particular, the Definition \ref{definition: recursive shift systems and nested FT constructions in our cases} of a recursive shift system enables us to apply the shift system theory (Section \ref{subsection: recursive shift systems}) by the ``univariation" at each recursive step, and by applying this to the end along the procedure in Section \ref{subsection: derivation of Z-invariants}, we obtain the nested FT constructions \eqref{nested Feigin-Tipunin construction in the main text} of the ``simple modules of the triplet Virasoro VOA" above.
The recursive combinatorial derivation of $\hat{Z}$-invariants in Section \ref{subsection: derivation of Z-invariants} is regarded as the nested FT construction at the level of the Grothendieck group, and \eqref{character of general singlet module} is the Weyl-type (or bosonic) character formula of \eqref{nested Feigin-Tipunin construction in the main text}. 
On the other hand, the above reduction of variables occurs not only by defragmentation, but also by setting $r_i=p_i$ or by reducing $N$.
The former is the state in which the Felder complex with respect to that variable is degenerate, and in the most extreme case $(\forall r_i=p_i)$, the simple module becomes projective.
However, even in such a degenerate case, the (recursive) shift system theory can be considered for non-degenerate variables.

\end{document}